%% file: perturbaBlas_final.tex
\documentclass[a4paper,11pt, reqno, 
]{amsart}

\usepackage[english]{babel}
\usepackage[utf8]{inputenc}

\usepackage{graphicx}
\usepackage{amscd}
\usepackage{amsfonts}
\usepackage{amsaddr}
\usepackage{latexsym}
\usepackage{amsthm}
\usepackage{amssymb,amsmath}
\usepackage{enumerate}
\usepackage{verbatim}

\usepackage{import}

\usepackage[titletoc,title]{appendix}

\usepackage[usenames]{color}

\usepackage{hyperref}
\hypersetup{colorlinks=false}

\usepackage{subfigure}

\title[Sing.\ pert.\ of {B}laschke prod. and connectivity of Fatou comp.]{\bf Singular perturbations of {B}laschke Products and connectivity of Fatou components}

\author[Jordi Canela]{Jordi Canela}
\thanks{The  author was partially supported by the  grant  346300 for IMPAN from the Simons Foundation and the matching 2015–2019 Polish MNiSW fund,  by RedIUM and MINECO (Spain) through the research network MTM2014-55580-REDT, and by the mathematics institute IMAC}
\address{Inst. Univ. de Matemàtiques i Aplicacions de Castelló (IMAC)\\
Universitat Jaume I\\
Av. de Vicent Sos Baynat, s/n, 12071 Castelló de la Plana, Spain}
\address{Inst. of Math. Polish Academy of Sciences (IMPAN)\\
ul. \'Sniadeckich 8, 00-656 Warszawa, Poland}
\email{canela@maia.ub.es}

\newtheorem{teor}{Theorem} [section]
\newtheorem{thm}[teor]{Theorem}
\newtheorem{propo}[teor]{Proposition}

\newtheorem{lemma}[teor]{Lemma}

\newtheorem*{teoremB}{Theorem B}
\newtheorem*{teoremA}{Theorem A}

\theoremstyle{definition}

\newsavebox{\savepar}

\newcommand{\com}{\mathbb{C}}
\newcommand{\wcom}{\widehat{\mathbb{C}}}

\newcommand{\real}{\mathbb{R}}

\newcommand{\nat}{\mathbb{N}}

\newcommand{\dis}{\mathbb{D}}
\newcommand{\cercle}{\mathbb{S}^1}
\newcommand{\re}{\rm{Re}}
\newcommand{\im}{\rm{Im}}

\makeatletter
\def\blfootnote{\gdef\@thefnmark{}\@footnotetext}
\makeatother

    {\end{list}}

\begin{document}
\begin{abstract}
 {\noindent \small The goal of this paper is to study the family of singular perturbations of Blaschke products given by $B_{a,\lambda}(z)=z^3\frac{z-a}{1-\overline{a}z}+\frac{\lambda}{z^2}$. We focus on the study of these rational maps for parameters $a$ in the punctured disk $\dis^*$ and $|\lambda|$ small. We prove that, under certain conditions, all Fatou components of a singularly perturbed Blaschke product {$B_{a,\lambda}$} have finite connectivity but there are components of arbitrarily large connectivity within its dynamical plane. Under the same conditions we prove that the Julia set is the union of countably many Cantor sets of quasicircles and uncountably many point components.}
\end{abstract}

\maketitle



\section{Introduction}

Given a Rational map $f:\wcom\rightarrow \wcom$, where $\wcom$ denotes the Riemann sphere, the Fatou set $\mathcal{F}(f)$ is defined as the set of points $z\in\wcom$ such that the family of iterates $\{f(z), f^2(z)=f(f(z)),...\}$ is normal in some open neighbourhood of $z$. Its complement, the Julia set $\mathcal{J}(f)$, corresponds to the set of points where the dynamics is chaotic. The Fatou and the Julia sets are totally invariant under $f(z)$. The Fatou set is open and its connected components, known as Fatou components, map among themselves. The celebrated result of Sullivan \cite{Su} states that all Fatou components of rational maps are either periodic or preperiodic. Moreover, any cycle of periodic Fatou components has at least a critical point, i.e.\ a zero of $f'(z)$, somehow related to it (see \cite{Mi1}).

The aim of this paper is to study singular perturbations of a family of Blaschke products and analyse the structure of their dynamical plane. We focus on the special case for which Fatou components of arbitrarily large connectivity appear. The study of singular perturbations of rational maps is a very active research field in  holomorphic dynamics. They were used by McMullen to show the existence of buried Julia components for rational maps, i.e.\ connected components of the Julia set which do not intersect the boundary of any Fatou component. He described the nowadays called McMullen Julia sets,  Julia sets consisting of Cantor sets of quasicircles (see \cite{McM1}). A quasicircle is a Jordan curve which is the image of the unit circle by a quasiconformal map (c.f.\ \cite{BF}). To prove the existence of such Julia sets, he studied the singular perturbations of $p_n(z)=z^n$ given by

\begin{equation}\label{perturbedpolyn}
Q_{\lambda, n, d}(z)=z^n+\frac{\lambda}{z^d},
\end{equation}

\noindent where $\lambda\in\com$, $n\geq2$ and $d\geq1$. 
The map $p_n(z)$  has $z=0$ and $z=\infty$ as superattracting fixed points of local degree $n$. However, the singular perturbation performed on $Q_{\lambda, n, d}$ transforms $z=0$ on a preimage of $\infty$ and $n+d$ new critical points appear around it.  McMullen showed that the Julia set $\mathcal{J}(Q_{\lambda, n, d})$ is a Cantor set of quasicircles if $|\lambda|$ is small enough and $1/n+1/d< 1$ (see Figure~\ref{McMullenBasilica} (a)). Afterwards, Devaney, Look and Uminsky \cite{DLU} studied the different possible Julia sets of $Q_{\lambda, n, d}$ which may occur when all of the critical points belong to the basin of attraction of infinity. {They proved that}, in this situation, $\mathcal{J}(Q_{\lambda, n, d})$ can only be a Cantor set, a Cantor set of quasicircles or a Sierpinski curve (a homeomorphic image of a Sierpinski carpet). 
Since then, the family $Q_{\lambda, n, d}$ has  been the object of study of several other papers (see e.g.\  \cite{Mora}, \cite{QWY}, \cite{QRWY}).

\begin{figure}[p!]
    \centering
    
    \subfigure[\scriptsize{A Julia set of quasicircles.}]{
    \includegraphics[width=200pt]{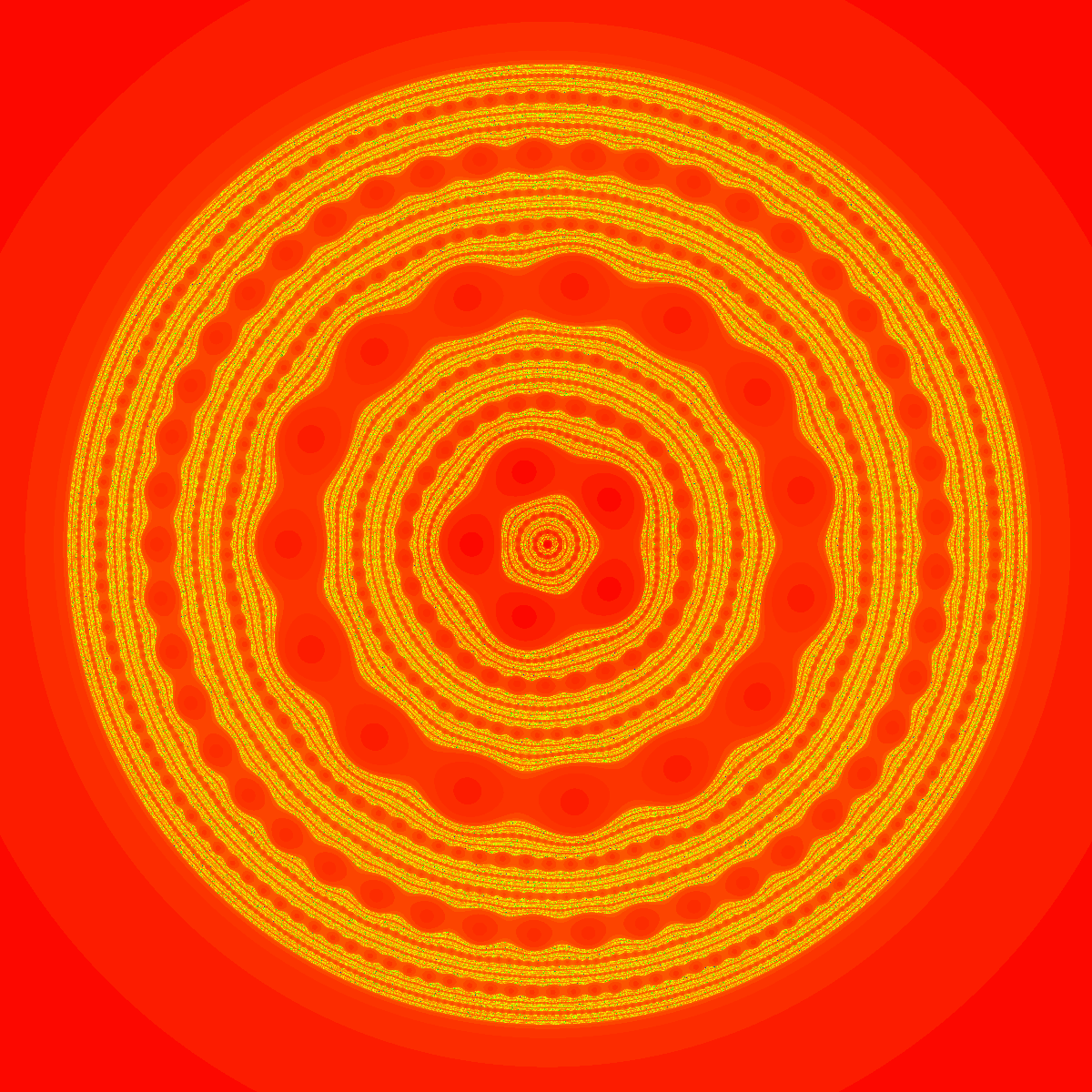}}
    \hspace{0.1in}
    \subfigure[\scriptsize{The Basilica. } ]{
    \includegraphics[width=200pt]{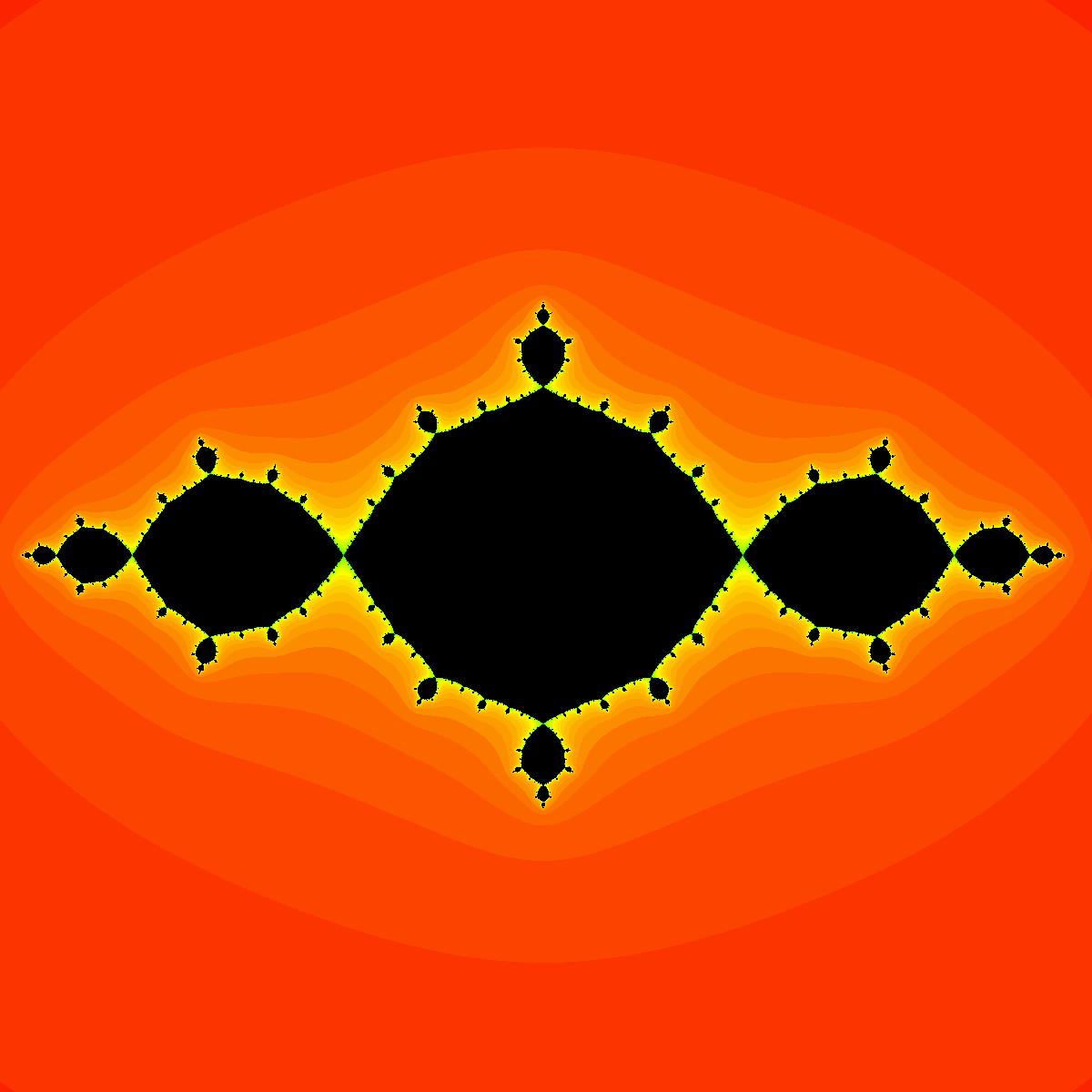}}

    \subfigure[\scriptsize{A perturbed Basilica.}]{
    \includegraphics[width=200pt]{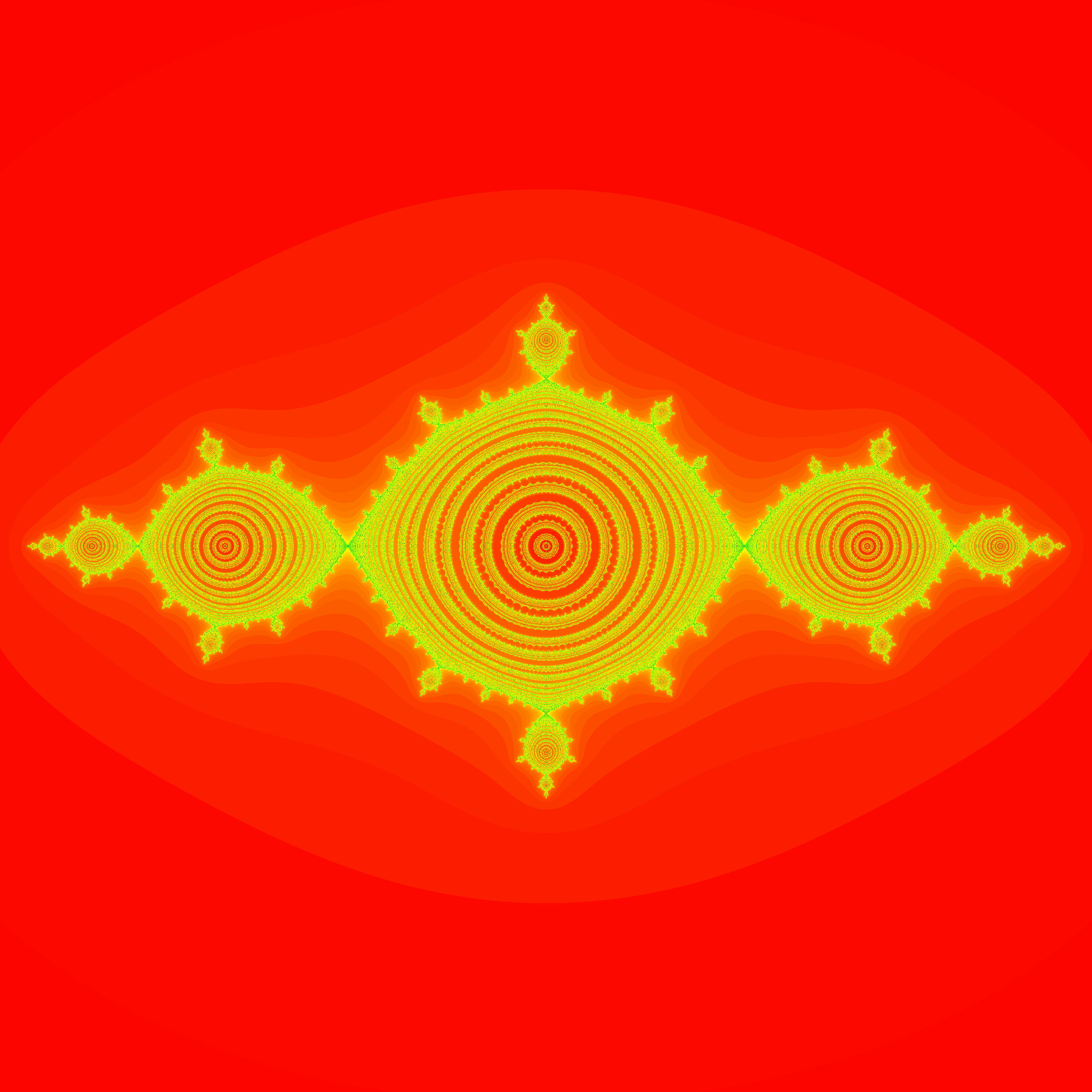}}
    \hspace{0.1in}
    \subfigure[\scriptsize{Zoom in Figure (c).}]{
    \includegraphics[width=200pt]{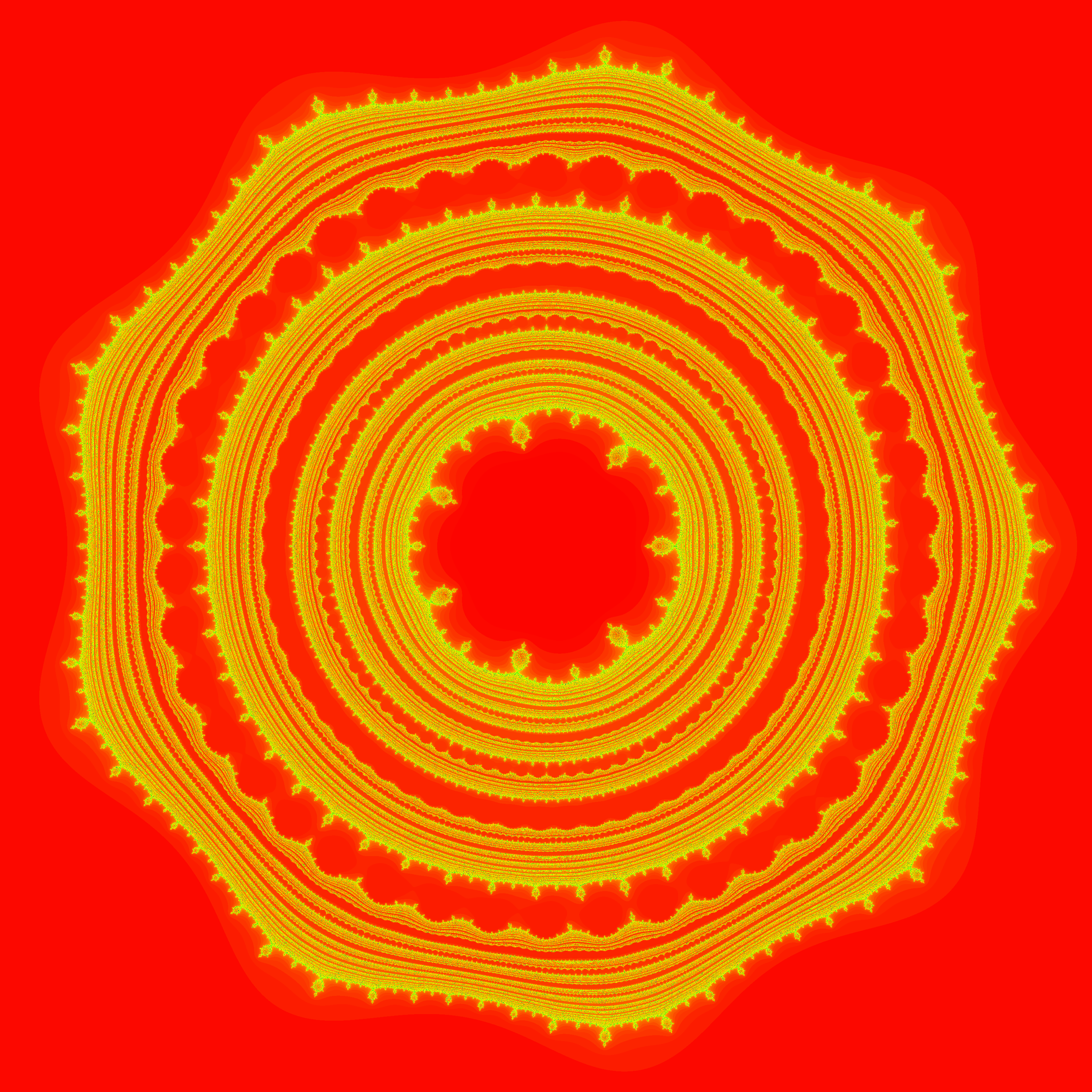}}
    
    \caption{\small Figure (a) corresponds to the dynamical plane of $Q_{\lambda,3,2}(z)=z^3+\lambda/z^2$ for $\lambda=10^{-4}$. Figure (b) corresponds to the dynamical plane of $p_{2,-1}(z)=z^2-1$, known as the Basilica.  The map {$p_{2,-1}$}  has a period 2 superattracting cycle at $\{0,-1\}$. Figure (c) corresponds to the dynamical plane of $f(z)= z^2-1+\lambda/(z^7(z+1)^5)$ for $\lambda=10^{-22}$. This map is a singular perturbation of the polynomial {$p_{2,-1}$} which adds a pole at each point of the superattracting cycle. Figure (d) is a magnification of (c) around the point $z=0$.  The colours are as follows. We use a scaling from yellow to red to plot the basin of attraction of $z=\infty$. In Figure (b) we plot the basin of attraction of the cycle $\{0,-1\}$ in black. In the other figures we may observe an approximation of the Julia set in yellow.  }

    \label{McMullenBasilica}
\end{figure}

It arises as a very interesting topic to investigate McMullen-like Julia sets, i.e.\ Julia sets of rational maps which are somehow similar to the ones described by McMullen. They may be obtained by adding one or several poles  to a polynomial $p_{n,c}(z)=z^n+c$ (see \cite{BDGR} and \cite{GMR}, respectively), where $n\geq 2$ and $c\in\com$ is such that $z=0$ belongs to a superattracting cycle. In this scenario, the structure of the Julia set of the polynomial {$p_{n,c}$} plays a very important role after the perturbation. The set $\mathcal{J}(p_{n,c})$ coincides with the boundary of the immediate basin of attraction of $z=\infty$, $A_{p_{n,c}}^*(\infty)$. After the perturbation, this remains as a component of the Julia set which moves continuously with respect to $\lambda$. Moreover,  many preimages of it appear (see Figure~\ref{McMullenBasilica} (b), (c) and (d)). Under certain conditions, this may lead to the existence of Cantor sets of closed curves instead of quasicircles (c.f.\  \cite{GMR}).

In this paper we introduce a family of rational maps with McMullen-like Julia sets which present different dynamics than the one of the previously mentioned works. To do so, we consider singular perturbations of rational maps whose Julia set is the unit circle and which have free critical points, i.e.\ critical points which do not belong to a superattracting cycle. The existence of these free critical points allows the appearance of other types of dynamics in the Fatou set. More specifically, we study a  family of singular perturbations of the Blaschke products
 
 \begin{equation}\label{blasequation}
B_{a}(z)= z^3\frac{z-a}{1-\overline{a}z},
\end{equation}

\noindent where $a\in\dis^*=\dis\setminus\{0\}$. These Blaschke products have $z=0$ and $z=\infty$ as superattracting fixed points. The basin of attraction of $z=0$ is given by $A(0)=\dis$ while  the basin of attraction of $z=\infty$ is the Riemann sphere minus the closed unit disk,  $A(\infty)=\wcom\setminus\overline{\dis}$. Consequently, the Julia set of these maps is the unit circle, $\mathcal{J}(B_a)=\cercle$. Moreover, they have two free critical points, $c_-(a)\in\dis$ and $c_+(a)\in\com\setminus\overline{\dis}$, which play an important role after the perturbation. See Section~\ref{sectionblas} for a more detailed introduction to their dynamics.

 We consider singular perturbations of the Blaschke products $B_a$ given by

\begin{equation}\label{perturbedblasequation}
B_{a,\lambda}(z)=z^3\frac{z-a}{1-\overline{a}z}+\frac{\lambda}{z^2},
\end{equation}

\noindent where $\lambda\in\com$ and $a\in\dis^*$. If $\lambda=0$, they coincide with the Blaschke products $B_a(z)$ (Equation~(\ref{blasequation})).
If $\lambda\neq 0$ then the point $z=0$ becomes a pole and there appear $5$ new critical points and $5$ new zeros around it. The fixed point $z=\infty$ remains as a superattracting fixed point and there are two free critical points $c_+(a,\lambda)$ and $c_-(a,\lambda)$ which come from analytic continuation of the critical points $c_+(a,0)=c_+(a)$ and $c_-(a,0)=c_-(a)$ of $B_{a,0}(z)=B_a(z)$. For $|\lambda|$ small enough, the dynamics of $B_{a,\lambda}(z)$ around $z=0$ are similar to the ones of $Q_{\lambda, 3, 2}$. Since the maps $Q_{\lambda, n, d}$ present McMullen Julia sets for $|\lambda|$ small enough if $1/n+1/d<1$ and we have that $1/3+1/2<1$, we may a priori expect some sort of McMullen-like Julia set for the maps {$B_{a,\lambda}$}.  On the other hand, the free critical points allow the existence of other types of Julia sets than the ones described by Devaney, Look and Uminsky in \cite{DLU}. However, the lack of symmetry of the critical points of {$B_{a,\lambda}$} makes it much more difficult to investigate the family {$B_{a,\lambda}$} than the family {$Q_{\lambda, n, d}$} from a global point of view. We want to point out that similar results to the ones presented in this paper could be obtained if replacing the powers 3 and 2 in Equation~(\ref{perturbedblasequation}) by naturals $n$ and $d$ such that $1/n+1/d<1$.

The goal of this paper is to study the dynamics of the family {$B_{a,\lambda}$} for $a\in\dis^*$ and $|\lambda|$ small, and to show that this family provides examples of new phenomena related to the connectivity of Fatou components. More precisely, it is known that any periodic Fatou component has connectivity   1, 2 or $\infty$ (c.f.\ \cite{Bear}), while preperiodic Fatou components can have finite connectivity greater than 2. Beardon \cite{Bear} introduced an explicit family of rational maps suggested by Shishikura with a Fatou component of finite connectivity greater than 2. This family was studied more deeply in \cite{GL}. The authors proved that if the parameter is small enough then there are Fatou components of connectivity 3 and 5. They also showed the existence of maps with a Fatou component of connectivity 9 and conjectured that for any given $n\in\nat$, there exists a map within the family with a Fatou component of connectivity greater than $n$. Baker, Kotus and L{\"u} \cite{BKL} used a quasiconformal surgery procedure to show that, given any $n\in\nat$, there exist rational and meromorphic maps with preperiodic Fatou components of connectivity $n$. Later on some explicit examples of rational maps with such properties were introduced (see \cite{QJ} and \cite{Sti}). However, the examples presented in \cite{BKL}, \cite{QJ} and \cite{Sti} use an increasing number of critical points to guarantee the existence of such multiply connected Fatou components. Therefore, the degree of the rational maps provided as examples grows with $n$. In Theorem~A we prove that, fixed $a\in\dis^*$, if $\lambda\in\com^*=\com\setminus\{0\}$, $|\lambda|$ is mall enough, {and $c_-(a,\lambda)$ belongs to the basin of attraction of $z=\infty$, $A(\infty)$,} then there are only three possibilities: either every Fatou component has connectivity less or equal than two, or less or equal than three, or one can find components of arbitrarily large finite connectivity.  Consequently, this family of degree 6 rational functions may contain maps with preperiodic Fatou components of arbitrarily large finite connectivity within a single dynamical plane. We do not know of any previous example having this property. In Figure~\ref{dynamfigureAB} and Figure~\ref{dynamfigureC} we provide numerical evidence that all 3 cases of Theorem~A take place for different parameters. {A rigorous proof of the existence of parameters for which all three cases hold is work in progress and will appear in a manuscript which is currently under preparation.}

\begin{teoremA}

Fix $a\in\dis^*$ and let $\lambda\in\com^*$. There exists a constant $\mathcal{C}(a)$ such that if $|\lambda|<\mathcal{C}(a)$ and $c_-(a,\lambda)\in A(\infty)$, then $c_-(a,\lambda)$ belongs to a connected component $\mathcal{U}_c$ of $A(\infty)\setminus A^*(\infty)$ and exactly one of the following holds.

\begin{enumerate}[a)]
\item The Fatou component $\mathcal{U}_c$ is simply connected. All Fatou components of $B_{a, \lambda}$ have connectivity 1 or 2 (see Figure~\ref{dynamfigureAB} (a) and (b)).
\item The Fatou component $\mathcal{U}_c$ is multiply connected and does not surround $z=0$. All Fatou components of $B_{a, \lambda}$ have connectivity 1, 2 or 3 (see Figure~\ref{dynamfigureAB} (c) and (d)).
\item The Fatou component $\mathcal{U}_c$ is multiply connected and surrounds $z=0$. All Fatou components of $B_{a, \lambda}$ have finite connectivity but there are components of arbitrarily large connectivity (see Figure~\ref{dynamfigureC}).
\end{enumerate} 
\end{teoremA}

\begin{figure}[p]
    \centering
    
    \subfigure[\scriptsize{Dynamical plane of a map {$B_{a,\lambda}$} for which statement \textit{a)} of Theorem~A holds.}]{
    \includegraphics[width=200pt]{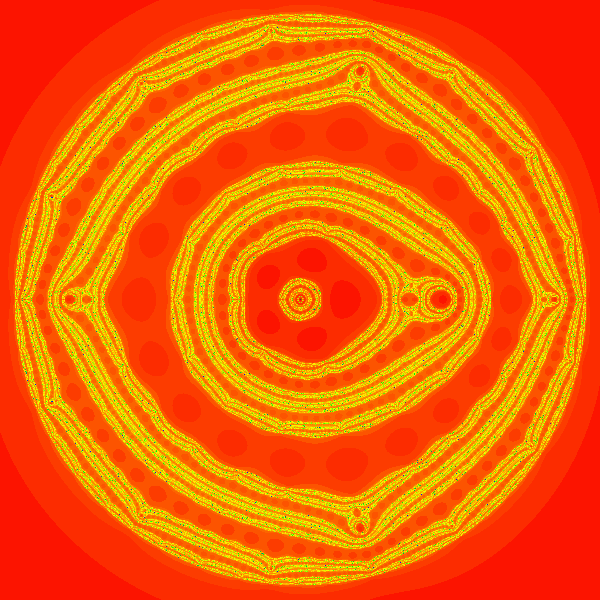}}
    \hspace{0.1in}
    \subfigure[\scriptsize{Zoom in Figure (a). } ]{
    \def\svgwidth{200pt}
    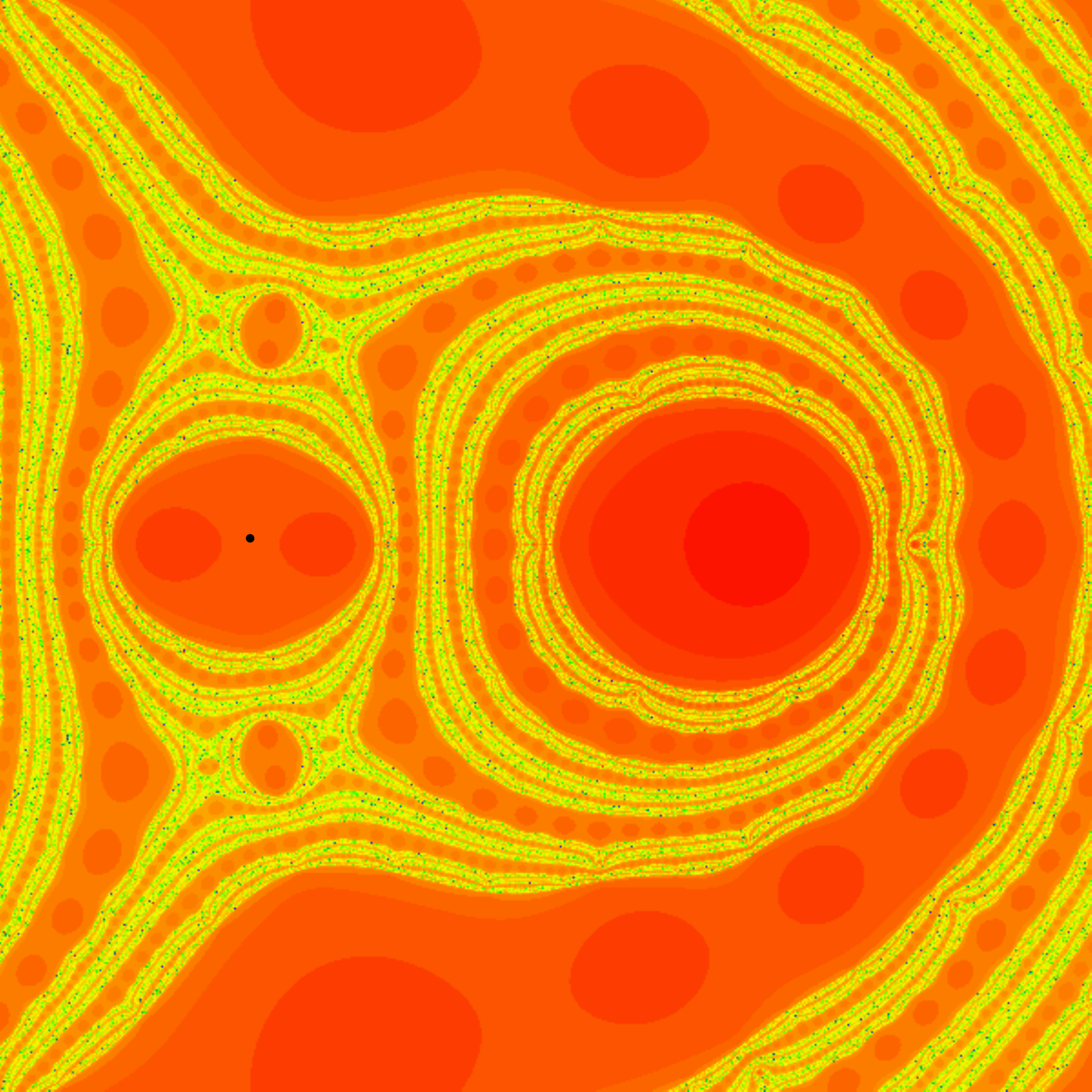}

    \subfigure[\scriptsize{Dynamical plane of a map {$B_{a,\lambda}$} for which statement \textit{b)} of Theorem~A holds.}]{
    \includegraphics[width=200pt]{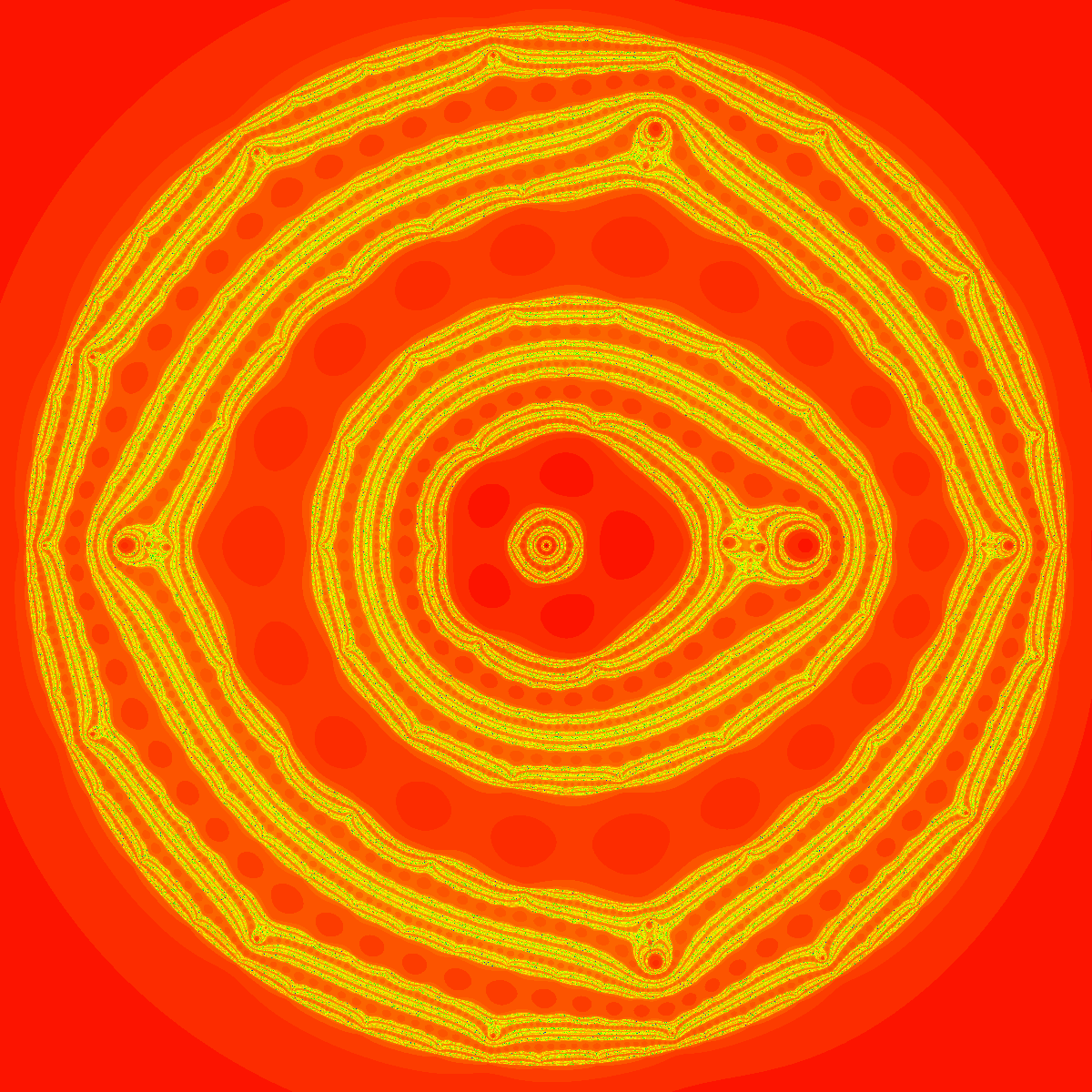}}
    \hspace{0.1in}
    \subfigure[\scriptsize{Zoom in Figure (c).}]{
    \def\svgwidth{200pt}
    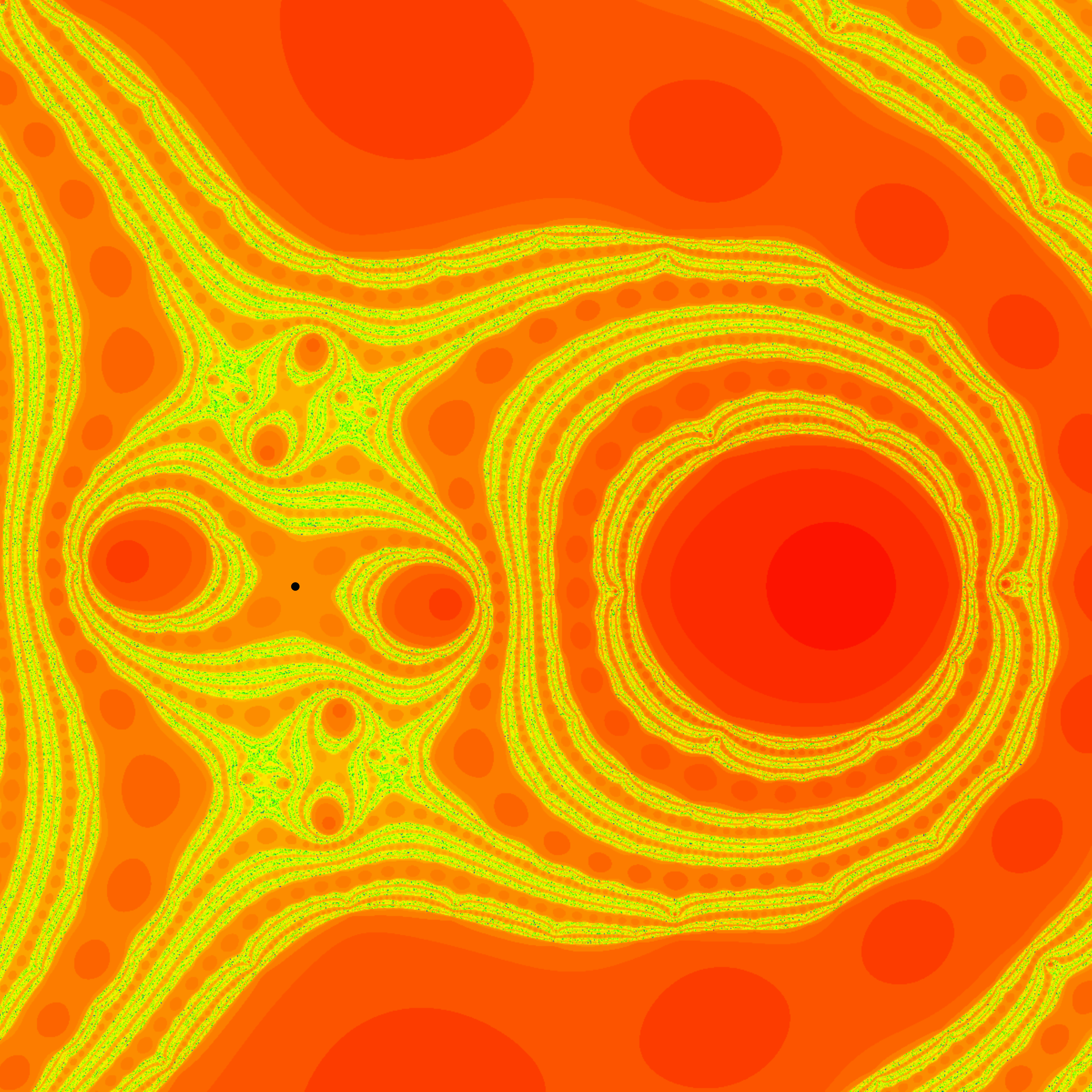}
    
    \caption{\small  Figures (a) and (b) represent the dynamical plane of the map {$B_{a, \lambda}$} where $a=0.5$ and $\lambda=3.022\times 10^{-5}$. Figures (c) and (d)  represent the dynamical plane of the map {$B_{a, \lambda}$} where $a=0.5$ and $\lambda=2.8\times10^{-5}+8.4\times 10^{-7}i $. These maps correspond to  singularly perturbed Blaschke products for which statements \textit{a)} and \textit{b)} of Theorem~A hold. The colours are as follows. We use a scaling from yellow to red to plot the basin of attraction of $z=\infty$. An approximation of the Julia set may be observed in yellow. }

    \label{dynamfigureAB}
\end{figure}

\begin{figure}[hbt!]
\centering
    \subfigure{
    \includegraphics[width=200pt]{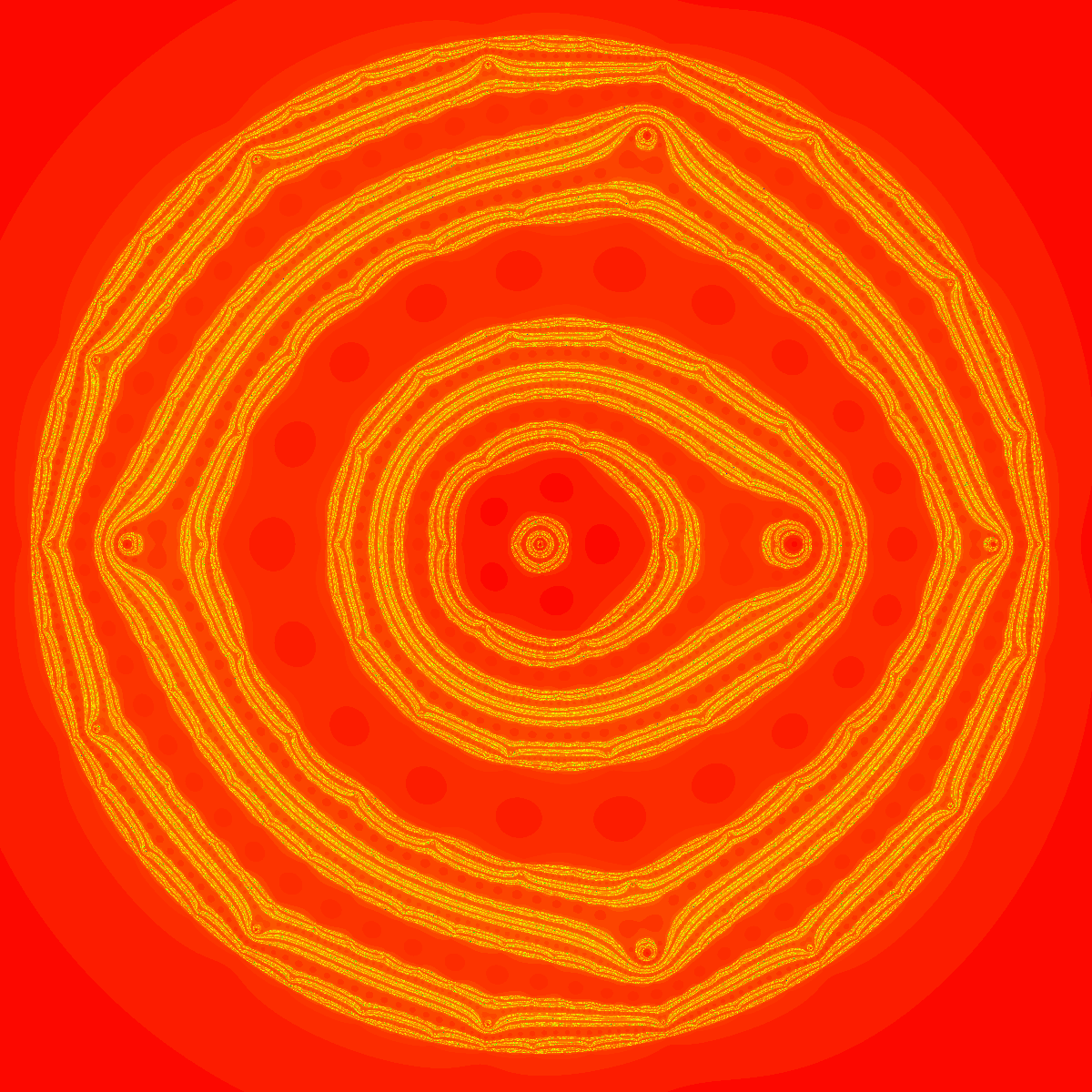}}
    \hspace{0.1in}
    \subfigure{
    \def\svgwidth{200pt}
    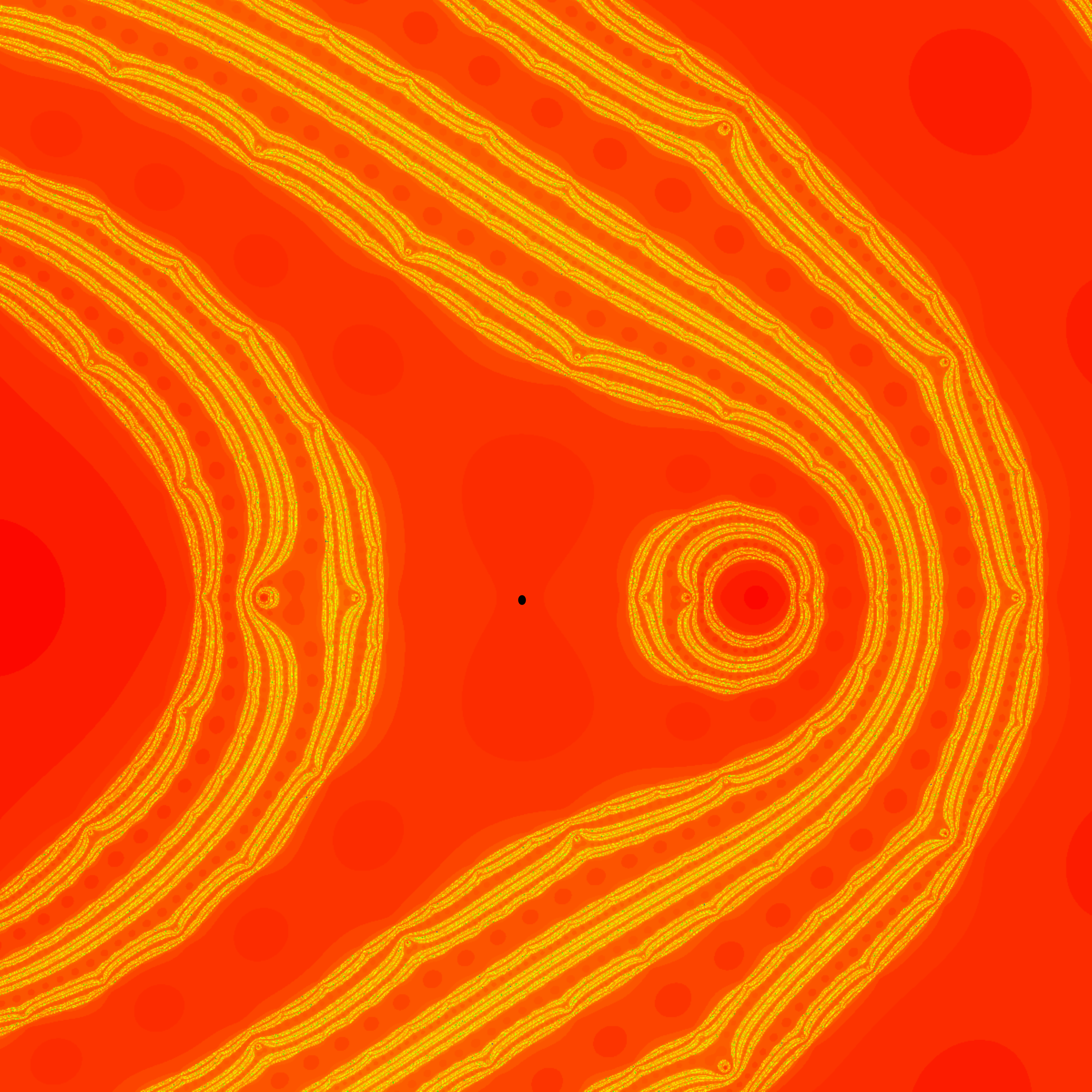}
    
    \caption{\small  The left corresponds to the dynamical plane of the map {$B_{a, \lambda}$} where $a=0.5$ and $\lambda=10^{-5}$. The right figure is a magnification of the left one. Statement \textit{c)} of Theorem~A holds for this map. Colours are as in Figure~\ref{dynamfigureAB}.  }

    \label{dynamfigureC}
\end{figure}

 Figure~\ref{parameterspace} shows the parameter space of {$B_{a,\lambda}$} for $a=0.5$ and $|\lambda|$ small. We observe hyperbolic components of different connectivities which presumably correspond to the three different cases in Theorem A. More specifically, numerical exploration shows that the red annular regions which surround $\lambda=0$ correspond to parameters for which statement \textit{c)} holds, that the red simply connected regions correspond to parameters for which statement \textit{a)} holds, and that the red annular regions which surround simply connected regions correspond to parameters for which statement \textit{b)} holds. A detailed study of the parameter plane of the family for $a\in\dis^*$ and $|\lambda|$ small is work in progress. 
 
\begin{figure}[hbt!]
\centering
    \subfigure{
    \includegraphics[width=200pt]{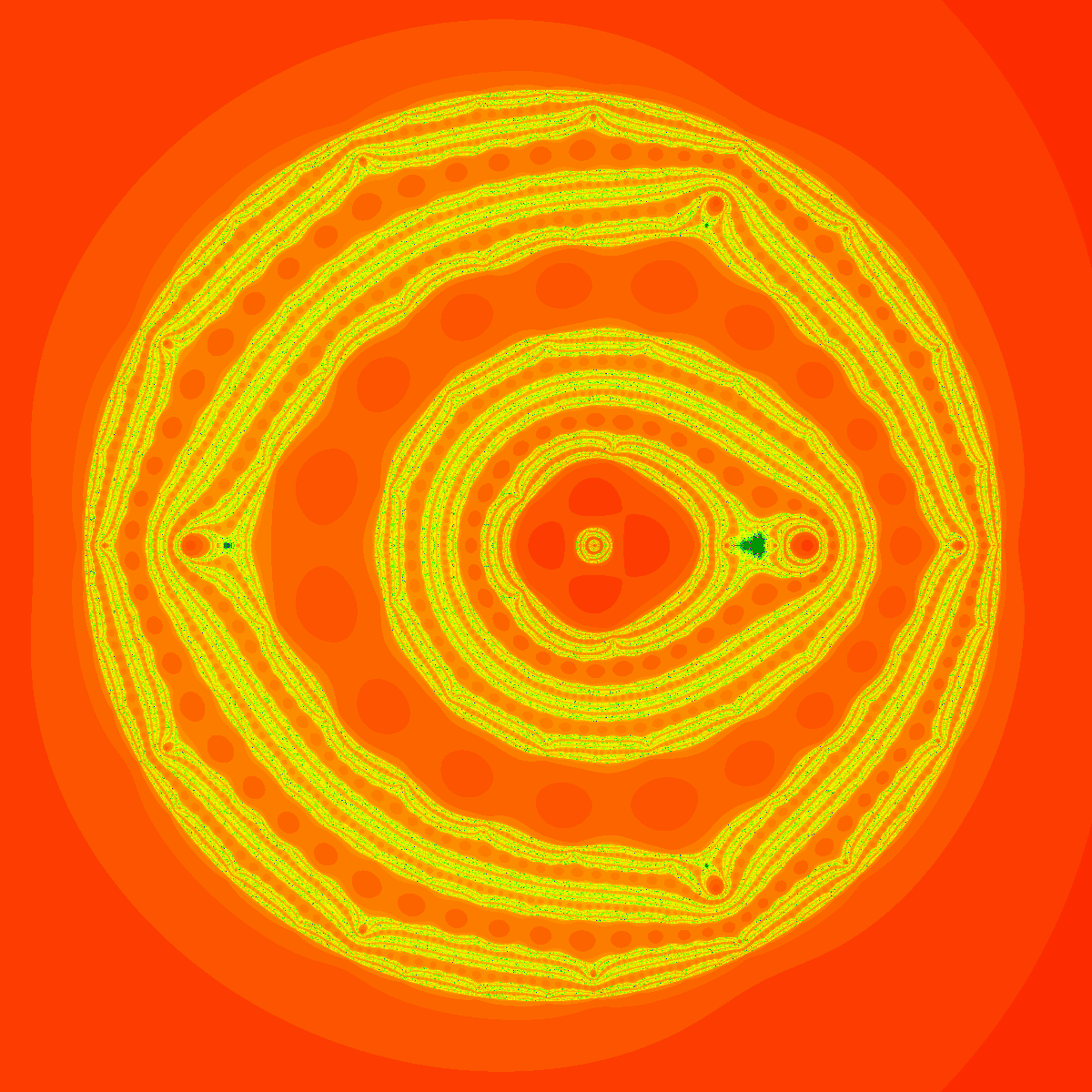}}
    \hspace{0.1in}
    \subfigure{
   \includegraphics[width=200pt]{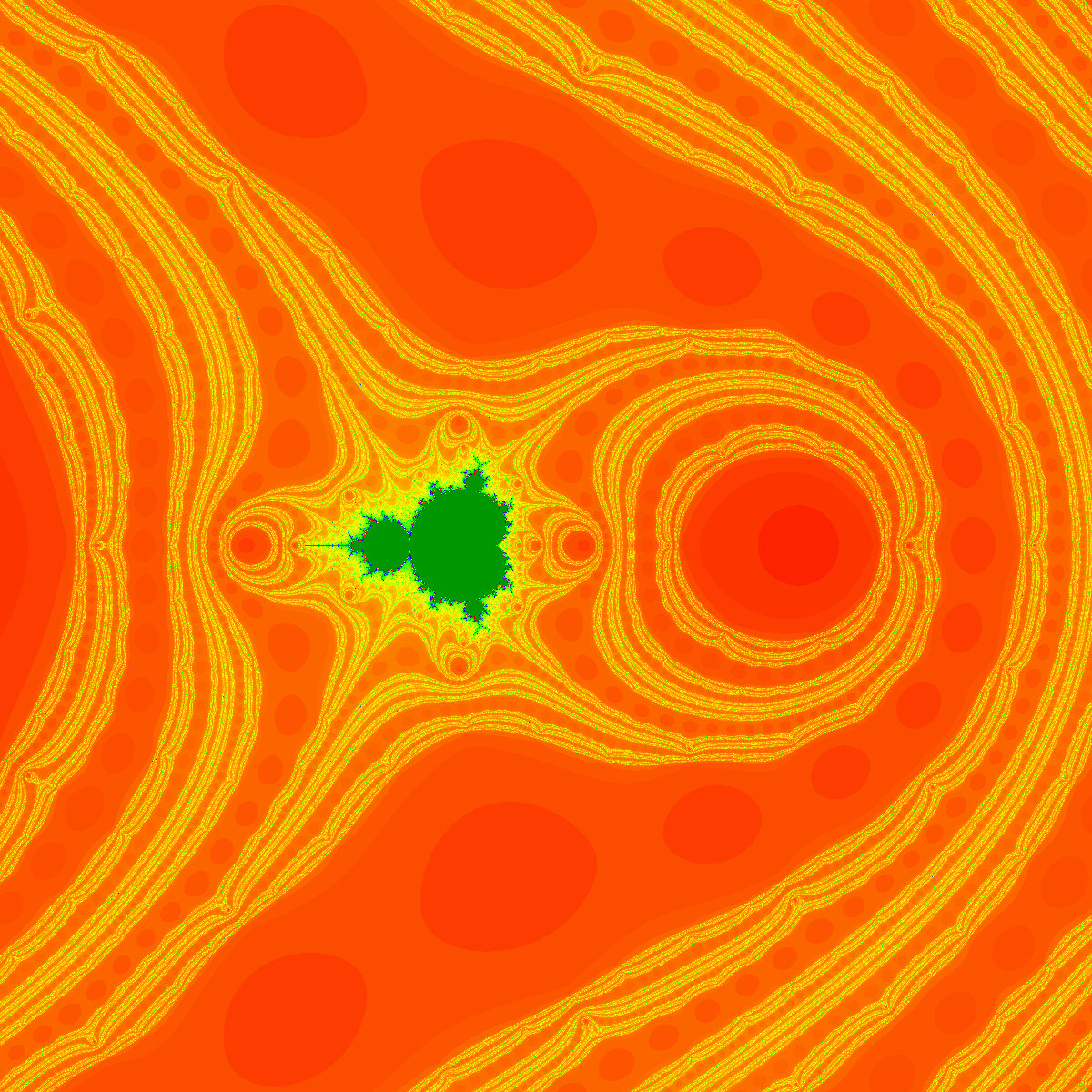}}
    
    \caption{\small The left figure corresponds to the parameter space of the family {$B_{a,\lambda}$} for $a=0.5$, $\re(\lambda)\in(-8.7\times 10^{-5}, 7.3\times10^{-5})$ and $\im(\lambda)\in(-8\times 10^{-5},8\times10^{-5})$. The right figure is a magnification of the left one. The colours are as follows. We use a scaling from yellow to red to plot parameters such that $c_-\in A(\infty)$ and green otherwise.   }

    \label{parameterspace}
\end{figure} 
 
The third case of Theorem A provides dynamical planes with other interesting properties. Our next theorem shows that Cantor sets of quasicircles  and uncountably many point components appear in the Julia set in these situations. Even if the result is similar to those in other examples in the literature (c.f.\ \cite{BDGR} and \cite{GMR}), we shall see in the proof that the dynamics of these maps are quite different.
 

\begin{teoremB}
Fix $a\in\dis^*$ and let  $\lambda\in\com^*$. Assume that $|\lambda|<\mathcal{C}(a)$, that $c_-(a,\lambda)\in A(\infty)$ and  that $c_-(a,\lambda)$ lies in a multiply connected Fatou component which surrounds $z=0$. Then the Julia set of {$B_{a,\lambda}$} contains a countable union of Cantor sets of quasicircles and uncountably many point components.

\end{teoremB}

We want to remark that if statements \textit{a)} or 	\textit{b)} of Theorem~A hold then the situation is essentially different. In this scenario, we can find annular Fatou domains bounded by quasicircles as described by McMullen in  \cite{McM1}. However, in this case they cannot lead to a Cantor set of quasicircles in the same way. This is because of the Fatou components which are homeomorphic to disks. They are to be found between McMullen's annuli together with other doubly (and triply) connected Fatou components (see Figure~\ref{dynamfigureAB}). This fact avoids McMullen's annuli to accumulate forming a Cantor set of quasicircles.

The paper is structured as follows. In Section~\ref{prelim} we describe the dynamics of the unperturbed Blaschke products {$B_a$}  and investigate the local dynamics of the maps $B_{a,\lambda}(z)$ around $z=0$ and the configuration of zeros and critical points (see Theorem~\ref{thmcritzeros}). In Section~\ref{sectionFatou} we study the Fatou components of the singularly perturbed Blaschke products and prove Theorem~A as a restatement of Theorem~\ref{thmtrichotomy}. Finally, in Section~\ref{sectionJulia} we investigate the Julia sets of the maps {$B_{a,\lambda}$} proving Theorem~B.

 \textit{Acknowledgements.} The author would like to thank X.~Jarque for the original idea which lead to this paper. He would also like to thank  N.~Fagella, A.~Garijo and the referee for their many and useful comments which greatly improved this paper.

\section{Preliminaries}\label{prelim}

The goal of this section is to understand the local dynamics which take place near $z=0$ for the singular perturbations of rational maps $B_{a,\lambda}(z)$ (Equation (\ref{perturbedblasequation})). In section~\ref{sectionblas} we study the dynamics of the Blaschke products $B_{a}(z)$ (Equation (\ref{blasequation})). In Section~\ref{localdynam} we prove Theorem~\ref{thmcritzeros}, which is the main result of the preliminaries. It describes the local dynamics around $z=0$ as well as the configurations of critical points of zeros. It also analyses the topology of the immediate basin of attraction of $z=\infty$, $A^*(\infty)$, and introduces notation which is used along the paper.

\subsection{Dynamics of the Blaschke products}\label{sectionblas}

We consider the family of Blaschke products of the form $B_{a}(z)=z^3(z-a)/(1-\overline{a}z)$  where $ z\in\wcom$ and $a\in\dis^*$.  The dynamics of these maps was studied in \cite{CFG1} and \cite{CFG2} for $a\in\com$.  As all Blaschke products, they leave the unit circle $\cercle$ invariant. Moreover,  $z=0$ and $z=\infty$ are superattracting fixed points of local degree 3 and, therefore, are critical points of multiplicity 2.
Given the fact that these rational maps have degree $4$, they have $6$ critical points counted with multiplicity. The other two  critical points, denoted by $c_{\pm}$, are given by

$$
c_{\pm}:=c_{\pm}(a):=a \cdot \frac{1}{3|a|^2}\left(2+|a|^2\pm\sqrt{(|a|^2-4)(|a|^2-1)}\right).
$$

If $a\in\dis^*$, then the critical points satisfy $c_-\in\dis$ and $c_+\in\com\setminus\overline{\dis}$. Moreover, since the only pole $z_{\infty}=1/\overline{a}$ does not lie in $\dis$, we have that $B_a(\dis)=\dis$ and all points in $\dis$ belong to the immediate basin of attraction of $z=0$, $A^*(0)$. Analogously we also have $A^*(\infty)=\wcom\setminus\overline{\dis}$.  We finish this first part of the preliminaries with Proposition~\ref{corbesBlaschke}, which will be used in the next subsection. During the proof we will use the Riemann-Hurwitz formula (c.f.\ \cite{Mi1, Ste}), which can be stated as follows. 

\begin{teor}[Riemann-Hurwitz formula]\label{riemannhurwitz}
Let $U$ and $V$ be two connected domains  of $\widehat{\com}$ of finite connectivity $m_{U}$ and $m_{V}$ and let $f:U\rightarrow V$ be a degree $k$ proper map branched over $r$ critical points counted with multiplicity. Then
$$m_{U}-2=k(m_{V}-2)+r.$$
\end{teor}

\begin{propo}\label{corbesBlaschke}
Fixed $a\in\dis^*$, there are analytic Jordan curves $\gamma_0$, $\gamma_0^{-1}$, $\gamma_{\infty}$ and $\gamma_{\infty}^{-1}$ such that:

\begin{enumerate}[a)]
\item The curves $\gamma_0$ and $\gamma_0^{-1}$ belong to $\dis$ and surround $z=0$. The curve $\gamma_0^{-1}$ is mapped onto $\gamma_0$ with degree 4 under {$B_a$}. Moreover, we have $\gamma_0\subset \rm{Int}(\gamma_0^{-1})$ and $z_0, c_-\in \rm{Int}(\gamma_0^{-1})$, where $\rm{Int}(\gamma_0^{-1})$ denotes the bounded component of $\com\setminus\gamma_0^{-1}$.

\item The curves $\gamma_{\infty}$ and $\gamma_{\infty}^{-1}$ belong to $\com\setminus\overline{\dis}$ and surround the unit disk $\dis$. The curve $\gamma_{\infty}^{-1}$ is mapped onto $\gamma_{\infty}$ with degree 4 under {$B_a$}. Moreover, we have $\gamma_{\infty}\subset \rm{Ext}(\gamma_{\infty}^{-1})$ and $z_{\infty}, c_+\in \rm{Ext}(\gamma_{\infty}^{-1})$, where $\rm{Ext}(\gamma_{\infty}^{-1})$ denotes the unbounded component of $\com\setminus\gamma_{\infty}^{-1}$.

\end{enumerate}
\end{propo}

\proof
We show how to obtain $\gamma_0$ and $\gamma_0^{-1}$. The curves $\gamma_{\infty}$ and $\gamma_{\infty}^{-1}$ can be obtained in a similar way. Let $U$ be the maximal domain of definition of the Böttcher coordinate of the superattracting fixed point $z=0$ (see \cite{Mi1}). Then $\partial U$ contains the critical point $c_-$. Moreover, there is an extra component $V$ of $B_a^{-1}(B_a(U))$ attached to the critical point $c_-$ which contains the zero $z_0$ (see Figure~\ref{corbesbottcher}).

To finish the proof it is enough to take an analytic Jordan curve $\gamma_0$ surrounding $\overline{B_a(U)}$ {such that $\gamma_0\subset U \setminus \overline{B_a(U)}$}. We want to show that any such curve satisfies that there is a unique component $\gamma_0^{-1}$ of $B_a^{-1}(\gamma_0)$, which surrounds the set $\overline{U \cup V}$ and is mapped with degree 4 onto $\gamma_0$. Let $A$ be the annulus bounded by $\gamma_0$ and $\cercle$ and let $A^{-1}=B_a^{-1}(A)$. Since $\cercle$ is completely invariant, $A^{-1}$ consists of a unique connected component. It follows from the Riemann-Hurwitz formula (Theorem~\ref{riemannhurwitz}) that $A^{-1}$ is also an annulus since $A$ contains no critical value (i.e.\ the image of a critical point) and, therefore, $A^{-1}$ contains no critical point. The annulus $A^{-1}$ is bounded by $\cercle$ and an analytic Jordan curve, say $\gamma_0^{-1}$, which is mapped onto $\gamma_0$ under $B_a$ and surrounds $\overline{U \cup V}$. Given that $B_a|_{A^{-1}}$ is proper,  {$B_a$} maps $A^{-1}$ onto $A$ with a certain degree $d$. Since $B_a|_{\cercle}$ has degree 4, we conclude this degree $d$ is precisely  $4$ and that {$B_a$} maps $\gamma_0^{-1}$ onto $\gamma_0$ with degree 4. Consequently, there can be no other preimage of $\gamma_0$.
\endproof

\begin{figure}[hbt!]
\centering
\def\svgwidth{240pt}
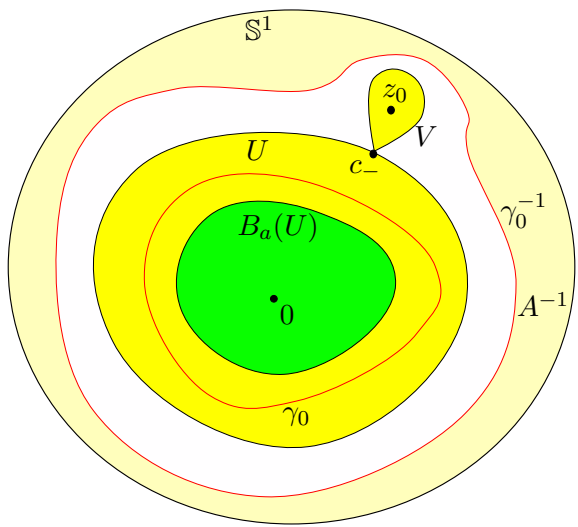
\caption{\small Scheme of the sets described in the proof of  Proposition~\ref{corbesBlaschke}. }
\label{corbesbottcher}
\end{figure}

\subsection{Local dynamics of the singular perturbations near the origin}\label{localdynam}

To understand the global dynamics of the singularly perturbed maps $B_{a,\lambda}(z)$ (Equation~(\ref{perturbedblasequation})), it is important to control the orbits of the critical points and to know where the preimages of $z=0$ and $z=\infty$ are. The goal of this section of the preliminaries is to locate them and use them to understand the local dynamics around $z=0$ for $|\lambda|$ small (see Theorem~\ref{thmcritzeros}).

As all degree 6 rational maps, the functions $B_{a,\lambda}$ have $2\cdot 6 -2 = 10$ critical points and 6 preimages of $z=0$ and $z=\infty$, counted with multiplicity. The point $z=\infty$ is a superattracting fixed point of local degree $3$, so it is a critical point of multiplicity $2$ and a preimage of itself of multiplicity $3$. Some of the other critical points, zeros and poles can be found by analytic continuation of the ones of $B_{a}(z)$ (Equation~(\ref{blasequation})). Indeed, if we fix $a\in\dis^*$, then the maps $B_{a,\lambda}(z)$ depend analytically on $\lambda$ for all $z$ in $\wcom\setminus \overline{\dis_{\epsilon}}$, where $\dis_{\epsilon}$ denotes the disc centred at 0 with radius $\epsilon$ as small as desired. Therefore, we have two free critical points $c_+(a,\lambda)$ and $c_-(a,\lambda)$, a zero $z_0(a,\lambda)$ and a pole $z_{\infty}(a,\lambda)=1/\overline{a}$ which are analytic continuation of the critical points $c_+$ and $c_-$, the zero $z_0$ and the pole $z_{\infty}$ of the Blaschke product $B_a(z)$. We shall drop the dependence on $a$ and $\lambda$ of $c_+(a,\lambda)$, $c_-(a,\lambda)$, $z_0(a,\lambda)$ and $z_{\infty}(a,\lambda)$ whenever it is clear from the context. 

There are 6 critical points, 5 zeros and 2 poles which are not to be found by analytic continuation of the ones of {$B_a$}. They appear in a small neighbourhood of $z=0$. Because of the term $\lambda/z^2$ of $B_{a,\lambda}(z)$, the point $z=0$ is a double preimage of infinity and a critical point of multiplicity 1. Therefore, there are only 5 zeros and 5 critical points whose location we have not described yet. We may approximate the values of these zeros and critical points using the fact that they are fixed points of certain operators (c.f.\ \cite{GMR}).

\begin{propo}\label{zeroscrit}
If we fix $a\in\dis^*$, then {$B_{a,\lambda}$} has 5 zeros of the form $\xi(\lambda/a)^{1/5}+o(\lambda^{1/5})$, where $\xi$ denotes a fifth root of the unity and $o(\lambda^{1/5})$ is such that $\lim_{\lambda\rightarrow 0} |o(\lambda^{1/5})|/|\lambda^{1/5}|=0$. Moreover, {$B_{a,\lambda}$} has  5 critical points  of the form $-\xi(2\lambda/3a)^{1/5}+o(\lambda^{1/5})$.
\end{propo}
\proof

We first prove the proposition for  the zeros of {$B_{a,\lambda}$}. They  are solutions of
$$z^3\frac{z-a}{1-\overline{a}z}+\frac{\lambda}{z^2}=0 \rightarrow z^5=-\lambda\frac{1-\overline{a}z}{z-a}.$$

For $\lambda=0$, the last equation has $z=0$ as a solution of multiplicity 5. If we increase $\lambda$, by continuity, we will have 5 zeros near $z=0$. These zeros are fixed points of the operators

$$T_{\xi}(z)=\xi\sqrt[5]{-\lambda\frac{1-\overline{a}z}{z-a}}=\xi\lambda^{1/5}\sqrt[5]{-\frac{1-\overline{a}z}{z-a}}=\xi\lambda^{1/5}R(z).$$

Notice that $R(z)$ does not depend on $\lambda$. We have 5 different choices for the operators $T_{\xi}$ given by the choice of the fifth root of the unity, denoted by $\xi$. Fix any of the five choices. Then one of the five zeros which appear around zero is a fixed point of $T_{\xi}$, say $w_{\lambda,\xi}$. We can approximate $w_{\lambda,\xi}$ by $T(0)=\xi(\lambda/a)^{1/5}$. Hence, we have

\begin{align*}
|w_{\lambda,\xi}-\xi(\lambda/a)^{1/5}|=|T(w_{\lambda,\xi})-T(0)|&{\leq \sup_{\eta\in\left[0,w_{\lambda,\xi}\right]}|T'(\eta)|\cdot|w_{\lambda,\xi}-0|}\\
&{= \sup_{\eta\in\left[0,w_{\lambda,\xi}\right]}|\lambda|^{1/5}|R'({\eta})|\cdot|w_{\lambda,\xi}|}.
\end{align*}

 {Moreover, we can find an upper bound of  $|R'({\eta})|$, $\eta\in\left[0,w_{\lambda,\xi}\right]$, for $|\lambda|$ small enough. Indeed, the function $R(z)$ does not depend on $\lambda$ and, if $|\lambda|$ is small enough, the points $z=a$ and $z=1/\overline{a}$  are bounded away from the segment $\left[0, w_{\lambda,\xi}\right]$ given that $w_{\lambda,\xi}\rightarrow 0$ as $\lambda\rightarrow 0$. Therefore, there is $\epsilon>0$ such that if $|\lambda|<\epsilon$,  then $\sup_{\eta\in\left[0,w_{\lambda,\xi}\right]}|R'(\eta)|<C$, where $C$ does not depend on $\lambda$. Finally,}

{$$\lim_{\lambda\rightarrow 0}\frac{|w_{\lambda,\xi}-\xi(\lambda/a)^{1/5}|}{|\lambda^{1/5}|}\leq  \lim_{\lambda\rightarrow 0}C\cdot|w_{\lambda,\xi}|=0.$$}

The proof for the  critical points is analogous using that they are solutions  of the equation

$$3z^2\frac{z-a}{1-\overline{a}z}+z^3\frac{1-|a|^2}{(1-\bar{a}z)^2}-2\frac{\lambda}{z^3}=0$$

\noindent and fixed points of the operators

$$T_{\xi}(z)=\xi(2\lambda)^{1/5}\sqrt[5]{\frac{1-\overline{a}z}{3(z-a)+z\frac{1-|a|^2}{1-\bar{a}z}}},$$

\noindent where $\xi$ denotes a fifth root of the unity. As before we can approximate the critical points by $T_{\xi}(0)=-\xi(2\lambda/3a)^{1/5}$.
\endproof

The next proposition follows directly from the previous result using the expressions of the zeros and critical points. The proof is straightforward. 

\begin{propo}\label{annuluszeroscrit}
Fix $a\in\dis^*$. Then, if $|\lambda|$ is small enough, the 5 critical points and the 5 zeros which appear around $z=0$ belong to the annulus of inner radius $\left(\frac{|\lambda|}{2|a|}\right)^{1/5}$ and outer radius $\left(2\frac{|\lambda|}{|a|}\right)^{1/5}$.
\end{propo}

The main result of the preliminaries is the following theorem. It describes the local dynamics that takes place near $z=0$ for $|\lambda|$ small. It also describes the topology of the immediate basin of attraction of infinity, $A^*(\infty)$. 

\begin{thm}\label{thmcritzeros}

Fix $a\in\dis^*$ and let $\lambda\in\com^*=\com\setminus\{0\}$. Then, there is a constant $\mathcal{C}(a)$ such that if $|\lambda|<\mathcal{C}(a)$  the following hold:

\begin{enumerate}[a)]
\item The immediate basin of attraction of $\infty$, $A^*(\infty)$,  is simply connected and $\partial A^*(\infty)$ is a quasicircle. Moreover, $A^*(\infty)$ is mapped with degree $4$ onto itself and contains only a pole $z_{\infty}$ and a critical point $c_+$ other than the superattracting fixed point $z=\infty$.

\item There is a simply connected neighbourhood $T_0$ of $z=0$ which is mapped 2 to 1 onto $A^*(\infty)$.

\item There is an open annulus $A_0$ which contains 5 critical points and 5 preimages of $z=0$ and is mapped 5 to 1 onto $T_0$.

\item The annular region in between $A_0$ and $ A^*(\infty)$ contains a critical point $c_-$ and a zero $z_0$. Moreover, the component $D_{0}$ of $B_{a,\lambda}^{-1}(T_0)$ in which  $z_0$ lies is simply connected and is mapped with degree 1 onto $T_0$. Consequently, it does not contain the critical point $c_-$.

\end{enumerate}
\end{thm}

\begin{figure}[hbt!]
\centering
\def\svgwidth{250pt}
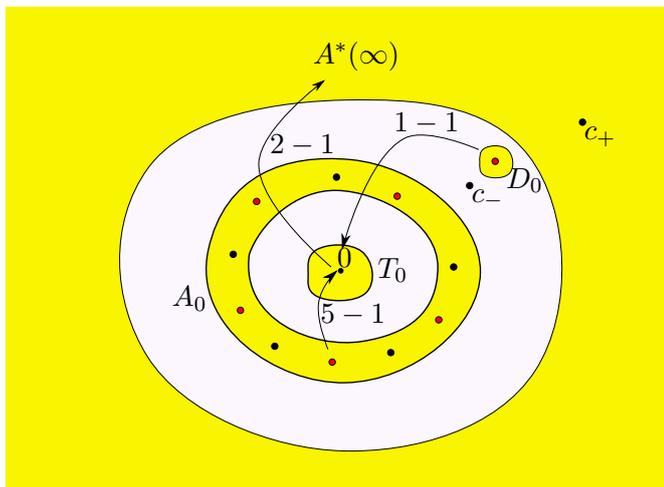
\caption{\small Scheme of the dynamics described in Theorem~\ref{thmcritzeros}. We draw in red the preimages of zero and in black the critical points. }
\label{esquema1}
\end{figure}


The conclusions of Theorem~\ref{thmcritzeros} are summarized in Figure~\ref{esquema1}. The proof of Theorem~\ref{thmcritzeros} is structured as follows.
We first prove statement \textit{a)}. Afterwards we give a criterion which guarantees that a point $z$ belongs to the preimage $T_0$ of $A^*(\infty)$ (Lemma~\ref{preiminfinit}). Finally we use the previous results and Proposition~\ref{annuluszeroscrit} to prove the remaining statements of Theorem~\ref{thmcritzeros}.

\begin{proof}[Proof of  statement a) of Theorem~\ref{thmcritzeros}]

{Let $a\in\dis^*$. We prove that if $|\lambda|$ is small enough, then statement \textit{a)} holds.} Take the curves $\gamma_0$ and $\gamma_{\infty}$ from Proposition~\ref{corbesBlaschke}. Then, it follows from continuity with respect to $\lambda$ that if $|\lambda|$ is small enough then we can take connected components $\gamma_{0,\lambda}^{-1}$ and $\gamma_{\infty,\lambda}^{-1}$ of $B_{a,\lambda}^{-1}(\gamma_0)$ and $B_{a,\lambda}^{-1}(\gamma_{\infty})$, respectively, such that the following hold:

\begin{itemize}
\item The sets  $\gamma_{0,\lambda}^{-1}$ and $\gamma_{\infty,\lambda}^{-1}$ are analytic Jordan curves.

\item The application {$B_{a,\lambda}$} maps  $\gamma_{0,\lambda}^{-1}$ and $\gamma_{\infty,\lambda}^{-1}$ onto $\gamma_0$ and $\gamma_{\infty}$ with degree 4, respectively.

\item The curves satisfy $\gamma_0\subset \rm{Int}\left(\gamma_{0,\lambda}^{-1}\right)$ and $\gamma_{\infty}\subset \rm{Ext}\left(\gamma_{\infty,\lambda}^{-1}\right)\subset A^*(\infty)$.

\item There is no zero, pole or critical point in the annular region bounded by $\gamma_{0,\lambda}^{-1}$ and $\gamma_{\infty,\lambda}^{-1}$. Moreover, {$B_{a,\lambda}$}  maps the annulus bounded by $\gamma_{0,\lambda}^{-1}$ and $\gamma_{\infty,\lambda}^{-1}$ onto the annulus bounded by $\gamma_{0}$ and $\gamma_{\infty}$ with degree $4$.

\item We have the inclusion $\rm{Ext}\left(\gamma_{\infty,\lambda}^{-1}\right)\subset A^*(\infty)$. Moreover, $\rm{Ext}\left(\gamma_{\infty,\lambda}^{-1}\right)$ contains the pole $z_{\infty}$ and the critical point $c_+$ and no other than $z=\infty$.

\end{itemize}

Indeed, all of these properties were satisfied by the curves obtained in Proposition~\ref{corbesBlaschke}. By continuity, all of them are still satisfied as long as the perturbation is small enough.

We can now apply a standard quasiconformal surgery procedure to obtain a quasiregular map $F(z)$ which agrees with $B_{a,\lambda}(z)$ on the annulus bounded by $\gamma_{0,\lambda}^{-1}$ and $\gamma_{\infty,\lambda}^{-1}$ and is quasiconformally conjugate to $z\rightarrow z^4$ inside the regions $\rm{Int}\left(\gamma_{0,\lambda}^{-1}\right)$ and $\rm{Ext}\left(\gamma_{\infty,\lambda}^{-1}\right)$. In this surgery we use the curves $\gamma_0$, $\gamma_{0,\lambda}^{-1}$, $\gamma_{\infty}$ and $\gamma_{\infty,\lambda}^{-1}$ to glue the dynamics of $z^4$ in $\rm{Int}\left(\gamma_{0,\lambda}^{-1}\right)$ and $\rm{Ext}\left(\gamma_{\infty,\lambda}^{-1}\right)$. We refer to \cite{BF} for an introduction to quasiconformal surgery and to \cite[proof of Theorem 7.4]{BF} for the details on how to glue the dynamics of $z^4$. After the surgery we obtain a quasiregular map $F(z)$ such that the following hold.

\begin{itemize}

\item The map $F(z)$ is conformally conjugate to $z\rightarrow z^4$ inside the regions $\rm{Int}\left(\gamma_{0}^{-1}\right)$ and $\rm{Ext}\left(\gamma_{\infty}^{-1}\right)$ and has $z=0$ and $z=\infty$ as superattracting fixed points of local degree 4.

\item The map $F(z)$ has topological degree 4 and no other critical point than $z=0$ and $z=\infty$ in the sense that it is bijective in a neighbourhood of any point other than $z=0$ or $z=\infty$. 

\item The map $F(z)$ is conjugate to a holomorphic map $f(z)$ via a quasiconformal map $\varphi(z)$ that fixes $0$ and $\infty$, i.e.\ $f(z) = \varphi^{-1}\circ F\circ\varphi(z)$.

\end{itemize}

Since the quasiregular map $F(z)$ has topological degree 4 and has $z=0$ and $z=\infty$ as superattracting fixed points of local degree 4, then also does the holomorphic function $f(z)$. Therefore, $f(z)$ is necessarily of the form $f(z)=b z^4$ where $b\in\com^*=\com\setminus\{0\}$. Consequently, the Julia set of $f(z)$ consists of a circle $S$ (of radius $(1/|b|)^{1/3}$) which is the common boundary of the basins of attraction of $z=0$ and $z=\infty$. The image under the quasiconformal map of this circle, $\varphi(S)$, necessarily belongs to the region bounded by $\gamma_{0,\lambda}^{-1}$ and $\gamma_{\infty,\lambda}^{-1}$ since its complement belongs, by construction, to the basins of attraction of $z=0$ and $z=\infty$ under the quasiregular map $F(z)$.

We finish the proof noticing that, by construction, the immediate basin of attraction of $\infty$ under the quasiregular map {$F$}, $A_F^*(\infty)$, coincides with the immediate basin of attraction of $\infty$ under {$B_{a,\lambda}$}, $A_{B_{a,\lambda}}^*(\infty)$, and that $\partial A_F^*(\infty)=\varphi(S)$. Since $\varphi(S)$ is the image of a circle under a quasiconformal map it is, by definition, a quasicircle. We conclude that  $A_{B_{a,\lambda}}^*(\infty)$ is a simply connected domain bounded by a quascircle.

\end{proof}

The next lemma provides a criterion which guarantees that a point $z$ belongs to a preimage of $A^*(\infty)$, for $|\lambda|$ small.

\begin{lemma}\label{preiminfinit}
Let $a\in\dis^*$ and $\lambda\in \com^*$. If $|\lambda|$ is small enough, then the following hold.
\begin{enumerate}[a)]
\item We have the inclusion $\{z\in\com;\; |z|>2\}\subset A^*(\infty)$.
\item If $|z|<\left(\frac{|\lambda|}{3}\right)^{1/2}$ then $B_{a,\lambda}(z)\in A^*(\infty)$.

\end{enumerate}
\end{lemma}
\proof
The first statement follows by continuity of $B_{a,\lambda}(z)$ with respect to $\lambda$ and the fact that when $\lambda=0$ the immediate basin of attraction of infinity consists of the complement of the closed unit disk (see Section~\ref{sectionblas}).

Now assume that $|z|<\left(\frac{|\lambda|}{3}\right)^{1/2}<1$. Then, using that $|z^3(z-a)/(1-\overline{a}z)|<1$ and that $|\lambda|/|z|^2>3$ we have:

$$|B_{a,\lambda}(z)|=\left|z^3\frac{z-a}{1-\overline{a}z}+\frac{\lambda}{z^2}\right|\geq \left|\left|\frac{\lambda}{z^2}\right|-\left|z^3\frac{z-a}{1-\overline{a}z}\right|\right|> 3-1 = 2.$$

Hence, $|B_{a,\lambda}(z)|>2$ and $B_{a,\lambda}(z)\in A^*(\infty)$ by statement \textit{a)}.
\endproof

 We can now finish the proof of Theorem~\ref{thmcritzeros}.

\begin{proof}[Proof of statements b), c) and d) of Theorem~\ref{thmcritzeros}]

 By Proposition~\ref{annuluszeroscrit} we know that if $|\lambda|$ is small enough, then the round annulus $A$ of inner radius $\left(\frac{|\lambda|}{2|a|}\right)^{1/5}$ and outer radius $\left(2\frac{|\lambda|}{|a|}\right)^{1/5}$ contains the 5 critical points and the $5$ zeros which appear around $z=0$ after the perturbation. We first show that, if $|\lambda|$ is small enough, then  $A$ is mapped under {$B_{a,\lambda}$} into {the Fatou component $T_0$ which contains the pole $z=0$}. Every point $w\in A$ has the form $w=b\lambda^{1/5}$ with $b\in\com$ such that $(1/2|a|)^{1/5}<|b|<(2/|a|)^{1/5}$. Using the Taylor expansion of $(z-a)/(1-\overline{a}z)$ around $z=0$ we have

$$B_{a,\lambda}(z)=-a z^3+O(z^4)+ \frac{\lambda}{z^2}\Rightarrow |B_{a,\lambda}(w)|\leq C|\lambda|^{3/5}+O(|\lambda|^{4/5}),$$

\noindent where $C$ is a bounded constant which can be taken independently of $b$ and $\lambda$.  We know from Lemma~\ref{preiminfinit} that if $|\lambda|$ is small enough and $|z|<\left(\frac{|\lambda|}{3}\right)^{1/2}$, then $B_{a,\lambda}(z)\in A^*(\infty)$. Since $|\lambda|^{3/5}$ tends to zero faster than $|\lambda|^{1/2}$, we can conclude that if $|\lambda|$ is small enough then $B_{a,\lambda}(A)\subset T_0$, where $T_0$ denotes the preimage of $A^*(\infty)$ that contains $z=0$.

Up to this point we have studied the dynamics of all critical points and zeros of {$B_{a,\lambda}$} but the critical point $c_-=c_-(a,\lambda)$ and the zero $z_0=z_0(a,\lambda)$, which come from continuation of the critical point $c_-(a,0)$ and the zero $z_0(a,0)$ of $B_a(z)$ (Equation~(\ref{blasequation})). By continuity, if $|\lambda|$ is small enough we can take a curve $\gamma$ such that it contains $c_-$, separates $A$ and $z_0$, and such that $B_{a,\lambda}(\gamma)\subset \com\setminus \{T_0\cup A^*(\infty)\}$. This implies that $c_-$ and $z_0$ belong to different Fatou components than $A$.

 We can now finish the proof. Take $\mathcal{C}(a)$ to be a positive real constant depending on $a$ such that if $|\lambda|<\mathcal{C}(a)$ then all previous considerations hold. Since $z=0$ is a double preimage of $\infty$ and there are no other poles outside $A^*(\infty)$, we know that the preimage $T_0$ of $A^*(\infty)$ is mapped with degree 2 to $A^*(\infty)$ under {$B_{a,\lambda}$}. By application of the Riemann-Hurwitz formula (Theorem~\ref{riemannhurwitz}), we can conclude that $T_0$ is simply connected  since it contains no other critical point than $z=0$.  This  proves statement \textit{b)}.
 
 We denote by $A_0$ the Fatou component which contains the round annulus $A$. We know that it contains exactly 5 critical points and 5 zeros. Since it contains exactly $5$ zeros, then $A_0$ is mapped onto $T_0$ with degree $5$. Since $T_0$ is simply connected, by application of the Riemann-Hurwitz formula the set $A_0$ is doubly connected. This proves statement \textit{c)}. 
 
 Finally, the annular region in between $A_0$ and $A^*(\infty)$ contains $c_-$ and $z_0$ because of the choice of $\gamma$. Moreover, the connected component $D_{0}$ of $B_{a,\lambda}^{-1}(T_0)$ which contains $z_0$ is mapped with degree $1$ onto $T_0$ since it only contains a preimage of $z=0$ counting multiplicities. Consequently, it can not contain the critical point $c_-$ and is simply connected. This proves statement \textit{d)}.

\end{proof}

\section{Fatou set of the singularly perturbed Blaschke products: proof of Theorem~A}\label{sectionFatou}

The goal of this section is to study the Fatou components of the singular perturbation {$B_{a,\lambda}$} in the case that $c_-\in A(\infty)$ and to prove Theorem~A. The main result of the section is Theorem~\ref{thmtrichotomy}. Theorem~A is a direct corollary of Theorem~\ref{thmtrichotomy}. Indeed, Theorem~A is a simpler restatement of Theorem~\ref{thmtrichotomy} which omits some technicalities. 

The main ingredients for the proof of Theorem~\ref{thmtrichotomy} are Theorem~\ref{thmcritzeros}, the Riemann-Hurwitz formula (Theorem~\ref{riemannhurwitz}) and Proposition~\ref{grauanells}. This proposition describes the dynamics of {$B_{a,\lambda}$} in the case that $|\lambda|<\mathcal{C}(a)$, that the orbit of $c_-$, $\mathcal{O}(c_-)=\{c_-, B_{a,\lambda}(c_-), B_{a,\lambda}^2(c_-), \cdots\}$, intersects the annulus $A_0$, and that $c_-$ belongs to a Fatou component which surrounds $z=0$ (see Figure~\ref{esquemaanells}). 

\begin{propo}\label{grauanells}
Fix $a\in\dis^*$ and let $\lambda\in\com^*$. If  $|\lambda|<\mathcal{C}(a)$  and $\mathcal{O}(c_-)$ intersects the annulus $A_0$, then the critical point $c_-$ lies in a triply connected Fatou component, say $\mathcal{U}_c$, {which is eventually mapped onto $A_0$}. Moreover, if $\mathcal{U}_c$ surrounds $z=0$ then the following hold. 
\begin{itemize}

\item The triply connected Fatou component $\mathcal{U}_c$ bounds a closed disk $\mathcal{V}_1$ which is mapped with degree 1 onto the closed disk bounded by $B_{a,\lambda}(\mathcal{U}_c)$, which contains $T_0$.
\item The closed annular region $\mathcal{V}_2$ in between $T_0$ and $A_0$ is mapped with degree 2 onto the annular region in between $T_0$ and $ A^*(\infty)$. 
\item The closed annular region $\mathcal{V}_3$ in between $A_0$ and $\mathcal{U}_c$ is mapped with degree $3$ onto the annular region $\mathcal{W}_3$ in between $T_0$ and $B_{a,\lambda}(\mathcal{U}_c)$. 
\item The closed annular region $\mathcal{V}_4$ in between $\mathcal{U}_c$ and $A^*(\infty)$ is mapped with degree $4$ onto the annular region $\mathcal{W}_4$ in between  $B_{a,\lambda}(\mathcal{U}_c)$ and $A^*(\infty)$. 

\item The Fatou component $\mathcal{U}_c$ is contained in the region $\mathcal{W}_4$.

\item  The Fatou component $\mathcal{U}_c$ is mapped with degree 4 onto the annulus $B_{a,\lambda}(\mathcal{U}_c)$. 

\end{itemize}

\end{propo}

\begin{figure}[hbt!]
\centering
\def\svgwidth{440pt}
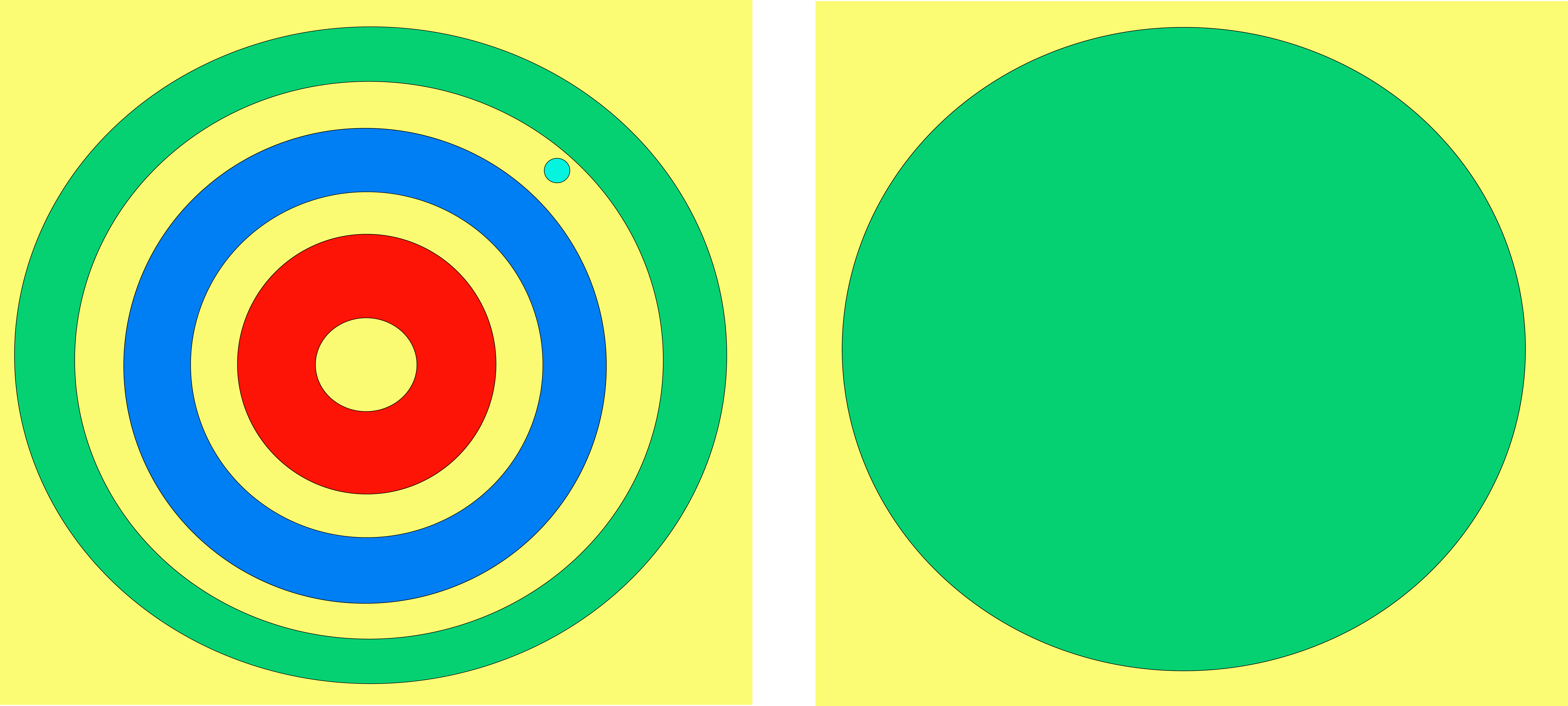
\caption{\small Summary of the dynamics of {$B_{a,\lambda}$} described in Proposition~\ref{grauanells}. The triply connected region $\mathcal{U}_c$ is mapped with degree 4 onto the annulus $B_{a,\lambda}(\mathcal{U}_c)$. The green annular region $\mathcal{V}_4$ is mapped with degree 4 to the green annular region $\mathcal{W}_4$.  The blue annular region $\mathcal{V}_3$ is mapped with degree 3 to the blue annular region $\mathcal{W}_3$. The pallid blue disk $\mathcal{V}_1$  is mapped with degree 1 to the region bounded by $B_{a,\lambda}(\mathcal{U}_c)$. The red region $\mathcal{V}_2$ is mapped with degree 2 to the full annular region bounded by $A^*(\infty)$ and $ T_0$.  {Since $\mathcal{U}_c\subset \mathcal{W}_4$, $\mathcal{V}_4\subset \mathcal{W}_4$ and either $B_{a,\lambda}(\mathcal{U}_c)=A_0$ or $B_{a,\lambda}(\mathcal{U}_c)\subset \mathcal{V}_3\cup \mathcal{V}_2$.  } }
\label{esquemaanells}
\end{figure}

\proof
It follows from the Riemann-Hurwitz formula (Theorem~\ref{riemannhurwitz}) that the preimage of an annulus $A$ which contains no critical value (i.e.\ the image of a critical point) under a holomorphic map is another annulus $A'$. Therefore the iterated preimages of the annulus $A_0$ are also annuli until one of them, which we denote by $\mathcal{U}_c$, contains the critical point $c_-$. Since $\mathcal{U}_c$ contains only one critical point and is mapped to a doubly connected domain with a certain degree $k$, it follows again from {the Riemann-Hurwitz formula that $\mathcal{U}_c$ has connectivity $k(2-2)+1+2=3$. Therefore, $\partial \mathcal{U}_c$ consists of 3 connected components. Moreover, since these boundary components are eventually mapped onto $\partial A^*(\infty)$, which is a quasicircle that intersects no critical orbit, they are quasicircles. 

Notice that the iterated preimages of $A_0$ either surround $z=0$ or some iterated preimage of $z=0$. Moreover, if a connected component $V$ of $B_{a,\lambda}^{-n}(A_0)$ does not surround $z=0$, then no component of $B_{a,\lambda}^{-1}(V)$ does. Now assume that $\mathcal{U}_c$ surrounds $z=0$. Since the boundary curves of $B_{a,\lambda}(\mathcal{U}_c)$ also surround $z=0$, we conclude that all boundary components of $\mathcal{U}_c$ surround $z=0$ or its preimage $z_0$ or both. Indeed, since $c_-$ lies in the annular region in between $A_0$ and $A^*(\infty)$, then so does $\mathcal{U}_c$ and the only zero that a connected component of $\partial \mathcal{U}_c$ can bound without surrounding $z=0$ is $z_0$. On the other hand, at most one component can bound both $z=0$ and $z_0$. Indeed, if two such boundary curves did, then the third one should lie in the annulus bounded by them but it could not surround neither $z_0$ nor $z=0$. Similar arguments yield that at most one such component surrounds $z=0$ and not $z_0$ and at most one of them surrounds $z_0$ and not $z=0$. Therefore, $\mathcal{U}_c$ has 2 exterior boundary components, which are quasicircles that surround $z=0$ (and consequently also the annulus $A_0$), and an interior boundary $\gamma_1$, which surrounds $z_0$ and not $z=0$ (see Figure~\ref{esquemaanells}). We denote by $\gamma_3$ and $\gamma_4$ the exterior boundary components, where $\gamma_3\subset\rm{Int}(\gamma_4)$.

Since the closed disk $\mathcal{V}_1$ bounded by $\gamma_1$ contains no pole and no critical point, it is mapped with degree 1 onto the closed disk bounded by $B_{a,\lambda}(\mathcal{U}_c)$ under {$B_{a,\lambda}$}. {Notice that $D_0\subset\mathcal{V}_1$, where $D_{0}$ is the Fatou component described in Theorem~\ref{thmcritzeros} which contains the zero $z_0$}.

Let $\mathcal{V}_4$ be the closed annular region in between $\mathcal{U}_c$ and $ A^*(\infty)$. Since $B_{a,\lambda}|_{\mathcal{V}_4}$ is proper, $\mathcal{V}_4$ is mapped onto the closed annular region $\mathcal{W}_4$ in between $B_{a,\lambda}(\mathcal{U}_c)$ and $ A^*(\infty)$ with a certain degree $d$ which is accomplished on the boundaries. Given that $\partial A^*(\infty)$ is mapped with degree 4 onto itself we can conclude that this degree is 4. 
Analogously, since $T_0$ is mapped with degree 2 onto $A^*(\infty)$, we can conclude that the annular region $\mathcal{V}_2$ in between $T_0$ and $A_0$ is mapped with degree 2 onto the annular region in between $T_0$ and $A^*(\infty)$. 

The domain $\mathcal{U}_c$ is contained inside the region $\mathcal{W}_4$ since otherwise the region $\mathcal{V}_4$ would be mapped into itself. Therefore, the points in $\mathcal{V}_4$ would never leave $\mathcal{V}_4$ under iteration. It would then follow from Montel's Theorem (see \cite{Mi1}) that the points in the interior of $\mathcal{V}_4$ would lie in a Fatou component which would not belong to $A(\infty)$, which is impossible since all critical points belong to $A(\infty)$.

The annular region $\mathcal{V}_3$ in between $A_0$ and $\mathcal{U}_c$ is mapped with degree 3 onto the region $\mathcal{W}_3$ in between  $T_0$ and $B_{a,\lambda}(\mathcal{U}_c)$. Indeed, $B_{a,\lambda}(\mathcal{V}_3)$ has 6 preimages since {$B_{a,\lambda}$} has degree 6 and it has 3 preimages missing since the sets $\mathcal{V}_1$ and $\mathcal{V}_2$ contain 1 and 2 preimages of it, respectively.

Finally, the domain $\mathcal{U}_c$ is mapped with degree 4 onto $B_{a,\lambda}(\mathcal{U}_c)$ since its boundaries are mapped with degree 4. Indeed, $\gamma_4$ is mapped with degree 4 onto one of the boundary components of $B_{a,\lambda}(\mathcal{U}_c)$ while $\gamma_1$ and $\gamma_3$ are mapped onto the other boundary component with degree 1 and 3, respectively.

\endproof

Using the information provided by Theorem~\ref{thmcritzeros} we can study the Fatou set of {$B_{a,\lambda}$} in the hyperbolic scenario for which the critical point $c_-$ belongs to $A(\infty)$.  Notice that $A_0$ and $D_{0}$ are the only preimages of $T_0$, which is the only preimage of $A^*(\infty)$. Therefore, a point in $A(\infty)\setminus\left(A^*(\infty) \cup T_{0}\right)$ is eventually mapped under iteration of {$B_{a,\lambda}$} into $A_0$ or into $D_{0}$. The Fatou domains that we may obtain depending on whether $c_-$ is eventually mapped into $A_0$ or $D_0$  can be very different and are studied in Theorem~\ref{thmtrichotomy}. In particular, Theorem~\ref{thmtrichotomy} tells us that if $c_-$ belongs to a preimage of $A_0$ which surrounds $z=0$ then all Fatou components have finite connectivity but there are components of arbitrarily large connectivity.

\begin{thm}\label{thmtrichotomy}
Fix $a\in\dis^*$ and let $\lambda\in\com^*$. Assume that $|\lambda|<\mathcal{C}(a)$ and that $c_-\in A(\infty)$. Then all Fatou components are bounded by quasicircles  and exactly one of the following is satisfied.

\begin{enumerate}[a)]
\item The critical point $c_-$ eventually falls under iteration into $D_{0}$ and  lies in a simply connected Fatou component. All Fatou components are either simply or doubly connected (see Figure~\ref{dynamfigureAB} (a) and (b)).

\item The critical point $c_-$ eventually falls under iteration into $A_0$ and lies in a triply connected Fatou component $\mathcal{U}_c$ which does not surround $z=0$.  All Fatou components have connecivity 1, 2 or 3 (see Figure~\ref{dynamfigureAB} (c) and (d)).
\item The critical point $c_-$ eventually falls under iteration into $A_0$ and lies in a triply connected Fatou component $\mathcal{U}_c$ which surrounds $z=0$. All Fatou components have finite connectivity but there are components of arbitrarily large connectivity (see Figure~\ref{dynamfigureC}).

\end{enumerate}

\end{thm}

\begin{proof}
Notice that all critical points belong to $A(\infty)$. Therefore, all Fatou components other than $A^*(\infty)$ are  preimages of it. Since $\partial A^*(\infty)$ is a quasicircle and no critical orbit intersects it we can conclude that the boundaries of all Fatou component are also quasicircles.

Since $c_-\in A(\infty)\setminus\left(A^*(\infty)\cup T_0\right)$, the only preimage of $A^*(\infty)$ is $T_0$ and the only preimages of $T_0$ are $A_0$ and $D_{0}$, we can conclude that $c_-$ is eventually mapped under iteration of {$B_{a,\lambda}$} into $D_{0}$ or into $A_0$.

First we assume that $\mathcal{O}(c_-)$ intersects $D_{0}$. Since all Fatou components other than $A^*(\infty)$, $T_0$, $D_{0}$ or $A_0$ are eventually mapped under iteration onto $A_0$ or $D_{0}$ we have that all of them are either disks or annulus. Indeed, since $\mathcal{O}(c_-)$ does not intersect $A_0$ we obtain from the Riemann-Hurwitz formula (Theorem~\ref{riemannhurwitz}) that all iterated preimages of $A_0$ are annuli. On the other hand, it also follows from the Riemann-Hurwitz formula that at least two different critical points are required to map a multiply connected domain onto a disk. Hence, all iterated preimages of $D_{0}$ are  disks.

Second we assume that $\mathcal{O}(c_-)$ intersects $A_{0}$. It follows from Proposition~\ref{grauanells} that there is a triply connected domain $\mathcal{U}_c$ which contains the critical point $c_-$. As in the previous case, every Fatou component which is not eventually mapped under iteration into $\mathcal{U}_c$ is either a disk or an annulus. If we also assume that $\mathcal{U}_c$ does not surround $z=0$ then all its iterated preimages are triply connected domains. To show that we can consider the filled disk $\mathcal{U}_c'$ which consists of the union of the points in $\mathcal{U}_c$ and the points contained in the bounded components of $\com\setminus\mathcal{U}_c$. Notice that all critical points are in $U=A^*(\infty)\cup T_0\cup A_0\cup \mathcal{U}_c$. Since $B_{a,\lambda}^{-n}(\mathcal{U}_c')\cap U=\emptyset$ for all $n\geq 1$, all branches of the inverse {$B_{a,\lambda}^{-n}$} restricted to $\mathcal{U}_c'$ are conformal for all $n\in\nat$. Consequently, all iterated preimages of $\mathcal{U}_c$ are also triply connected domains.

Finally we assume that  $\mathcal{O}(c_-)$ intersects $A_{0}$ and that $\mathcal{U}_c$ surrounds $z=0$. In that case, the domain $\mathcal{U}_c$ has a unique preimage in $\mathcal{V}_2$ and another one in $\mathcal{V}_3$ or $\mathcal{V}_4$, which also surround $z=0$. This follows from the fact that every Jordan curve surrounding $z=0$ and contained in the annular regions $\mathcal{W}_3$ or $\mathcal{W}_4$ has a unique preimage in the regions $\mathcal{V}_3$ or $\mathcal{V}_4$, respectively, and a unique  preimage in $\mathcal{V}_2$. Let $V$ be any of the two Fatou components which are preimages of $\mathcal{U}_c$ and surround $z=0$. Since {$B_{a,\lambda}$} restricted to  $\mathcal{V}_i$  has degree $i$, it follows from the Riemann Hurwitz-formula that $V$ has connectivity $i(3-2)+2>3$, where $V\subset \mathcal{V}_i$. Repeating this process we obtain a sequence of iterated preimages of $\mathcal{U}_c$ which surround $z=0$ and have connectivity $j(m-2)+2>m$, where $m$ is the connectivity of its image. In this way we obtain Fatou components of arbitrarily large connectivity (see Figure~\ref{dynamfigureC}).

\end{proof}

\section{Julia set of the singularly perturbed Blaschke products: proof of Theorem~B}\label{sectionJulia}

The goal of this section is to describe the Julia sets of the maps {$B_{a,\lambda}$} in the case that statement \textit{c)} of Theorem~\ref{thmtrichotomy} holds {and to prove} Theorem~B. To do so, we first show the existence of a Cantor set of quasicircles as described by McMullen in \cite{McM1}. Afterwards we show how to obtain countably many copies of the original Cantor set of quasicircles using the dynamics of the map. Finally we use symbolic dynamics to show the existence of an uncountable number of point components in the Julia set.

\begin{proof}[Proof of Theorem B]

We first show  the existence of a Cantor set of quasicircles. Let $A_0$ be the annulus containing the 5 critical points and 5 zeros which appear around $z=0$. It follows from the structure of the dynamical plane shown in Proposition~\ref{grauanells} that it has exactly two preimages which are Fatou components that surround $z=0$. One of them lies in $\mathcal{V}_2$ and the other one lies in $\mathcal{V}_3\cup \mathcal{V}_4\cup \mathcal{U}_c$. We can iterate this process, obtaining at each step exactly two new Fatou components which surround $z=0$ from every old component. These preimages are annuli, unless one of them is the triply connected Fatou component $\mathcal{U}_c$. In that case we repeat the process considering  the annulus $A_c'=\mathcal{U}_c\cup \mathcal{V}_1$ instead. In this fashion we obtain a sequence of open annuli whose boundaries are quasicircles which surround $z=0$ and belong to the Julia set. These boundaries together with its accumulation sets form a Cantor set of quasicircles in the same way as described by McMullen in \cite{McM1}.

We now see where the countable number of Cantor sets of quasicircles comes from. We know from Proposition~\ref{grauanells} that the closed disk $\mathcal{V}_1$ is mapped with degree one onto the closed disk bounded by $B_{a,\lambda}(\mathcal{U}_c)$. This last disk contains the part of the previously described Cantor set of quasicircles that lies in $\mathcal{W}_3$, which is also a Cantor set of quasicircles. Since the map {$B_{a,\lambda}$} restricted to $\mathcal{V}_1$ is a homeomorphism (it is conformal in a neighbourhood of $\mathcal{V}_1$), we conclude that $\mathcal{V}_1$ contains a Cantor set of quasicircles. To finish the argument it is enough to notice that the disk $\mathcal{V}_1$ has exactly 6 preimages  and so does every set which is eventually mapped onto $\mathcal{V}_1$. Therefore, $\mathcal{V}_1$ has countably many iterated preimages which are also homeomorphic to $\mathcal{V}_1$ since they cannot bound any critical point and, therefore, {$B_{a,\lambda}$} is conformal in a neighbourhood of them. Each of these disks contains a copy of the Cantor set in $\mathcal{V}_1$, which leads to the existence of a countable union of Cantor sets of quasicircles in the Julia set. Notice that these iterated preimages are not disjointed. Indeed, as we shall see in the remaining of the proof, many of these disks are contained in other ones.

The description that we have given of the Julia set up to this point is not complete. Indeed, it is known that periodic points are dense in the Julia set. However, the Cantor sets of quasicircles contained in the iterated preimages of $\mathcal{V}_1$ are eventually mapped into the original Cantor set {of quasicircles} and, therefore, contain no periodic point. The missing points of the Julia set are those whose orbit never {leaves} the union $\mathcal{D}$ of $\mathcal{V}_1$ and the disks which are eventually mapped onto it, $\mathcal{D}=\bigcup_{n=0}^{\infty}B_{a,\lambda}^{-n}(\mathcal{V}_1)$.
To understand the dynamics of those points we will use symbolic dynamics. We begin by labelling the filled annulus  $A_c'=\mathcal{U}_c\cup \mathcal{V}_1$ by $0$, $A_0':=A_c'$. Afterwards we label the iterated preimages of $A_c'$ which surround $z=0$ with the natural numbers, i.e.\ we denote every connected component of $\bigcup_{n=1}^{\infty}B_{a,\lambda}^{-n}(A_c')$ in $\mathcal{V}_2\cup \mathcal{V}_3\cup \mathcal{V}_4$ which surrounds $z=0$  by $A_i'$ with $i\in\nat$. With these labels we can associate to every point $z$ whose orbit never {leaves} $\mathcal{D}$ an itinerary $S(z)=(s_0,s_1,s_2,...)$, where $s_j=k$ if $B_{a,\lambda}^j(z)$ belongs  to $A_k'$. Notice that the itinerary depends on the chosen labelling, but this is not relevant.

It is now important to understand how the realizable itineraries, i.e.\ the itineraries which are realized  by some point $z\in\mathcal{D}$, look like. The fact that if $A_i'\neq A_0'$ then $B_{a,\lambda}(A_i')$ is another annulus, say $A_{s(i)}'$, implies that the index $i$ must always be followed by the index $s(i)$. By Proposition~\ref{grauanells} the disk $\mathcal{V}_1$, which is contained in $A_c'=A_0'$,  is mapped onto $\mathcal{W}_3\cup T_0$. Therefore, we conclude that the index $0$ must be followed by an index which labels an annulus contained in $\mathcal{W}_3$. Again by Proposition~\ref{grauanells}, since $\mathcal{U}_c\subset \mathcal{W}_4$ we conclude that we cannot have two zeros in a row.  Moreover, given any $i$, there is an $n=n(i)$ such that $s^n(i)=s\circ \overset{n)}{\cdots}\circ s(i)=0$.  Therefore, the index 0 appears infinitely many times in any realizable sequence $S(z)$.

We now discuss the simpler case where the sequence $S(z)$ is periodic. Without loss of generality we may assume that it begins with the index 0. Then we have $S(z)=(\overline{0,s_1,s_2,...,s_n})=(0,s_1,s_2,...,s_n,0,s_1,s_2,...,s_n,...)$. The first step is to understand the set of points which follow the sequence $(0,s_1,s_2,...,s_n,0)$. Since $\mathcal{V}_1$ does not intersect any critical orbit, every branch of the inverse $B_{a,\lambda}^{-n}$ is conformal in a neighbourhood of $\mathcal{V}_1$ for all $n\in\nat$. Given that {$B_{a,\lambda}^n$} maps $A_{s_1}'$ onto $A_0'$ with a certain degree $k\geq2$, there are $k$ disjoint preimages of $\mathcal{V}_1$ in $A_{s_1}'$.  Since {$B_{a,\lambda}$} maps $\mathcal{V}_1$ onto $T_0\cup \mathcal{W}_3\supset A_{s_1}'$ with degree 1, the set of points with itinerary $(0,s_1,s_2,...,s_n,0)$ corresponds to the disjoint union of $k$ sets which are compactly contained in $\mathcal{V}_1$. Notice that we can take small annuli in the Fatou set which surround these $k$ disks. Analogously, the set of points with itinerary $(0,s_1,s_2,...,s_n,0,s_1,s_2,...,s_n,0)$ correspond to $k^2$ disks compactly contained in the $k$ previous ones. Indeed, every original disk contains $k$ sub-disks, each of which is mapped univalently onto one of the original ones by {$B_{a,\lambda}^{n+1}$}.  As before, these new disks are surrounded by small annuli in the Fatou set. Repeating this process, at the step $p$ we obtain $k^p$ closed disks which are compactly contained in the $k^{p-1}$ disks of the previous step and are surrounded by small annuli in the Fatou set. Standard arguments of complex dynamics yield that these disks shrink to a Cantor set. We conclude that the set of points with itinerary  $S(z)=(\overline{0,s_1,s_2,...,s_n})$ corresponds to Cantor set of points. Moreover, since each of these points is surrounded by arbitrarily small annuli in the Fatou set we conclude that they are point components of the Julia set.

The general case derives analogously. The only difference is that every time that we come back to the index 0, the number of disks compactly contained in the previous ones varies depending on the subsequence. For each realizable sequence we obtain a nested sequence of disks which will also converge to a Cantor set of points which are point components of the Julia set (c.f.\ \cite{BDGR}).

\end{proof}

\bibliography{bibliografia}
\bibliographystyle{AIMS}

\end{document}

%% file: a05L3022e-5zoom.pdf_tex
\begingroup%
  \makeatletter%
  \providecommand\color[2][]{%
    \errmessage{(Inkscape) Color is used for the text in Inkscape, but the package 'color.sty' is not loaded}%
    \renewcommand\color[2][]{}%
  }%
  \providecommand\transparent[1]{%
    \errmessage{(Inkscape) Transparency is used (non-zero) for the text in Inkscape, but the package 'transparent.sty' is not loaded}%
    \renewcommand\transparent[1]{}%
  }%
  \providecommand\rotatebox[2]{#2}%
  \ifx\svgwidth\undefined%
    \setlength{\unitlength}{600bp}%
    \ifx\svgscale\undefined%
      \relax%
    \else%
      \setlength{\unitlength}{\unitlength * \real{\svgscale}}%
    \fi%
  \else%
    \setlength{\unitlength}{\svgwidth}%
  \fi%
  \global\let\svgwidth\undefined%
  \global\let\svgscale\undefined%
  \makeatother%
  \begin{picture}(1,1)%
    \put(0,0){\includegraphics[width=\unitlength,page=1]{a05L3022e-5zoom.pdf}}%
    \put(0.22380953,0.52704763){\color[rgb]{0,0,0}\makebox(0,0)[lb]{\smash{$c_-$}}}%
  \end{picture}%
\endgroup%

%% file: a05L28e-5i84e-7zoom.pdf_tex
\begingroup%
  \makeatletter%
  \providecommand\color[2][]{%
    \errmessage{(Inkscape) Color is used for the text in Inkscape, but the package 'color.sty' is not loaded}%
    \renewcommand\color[2][]{}%
  }%
  \providecommand\transparent[1]{%
    \errmessage{(Inkscape) Transparency is used (non-zero) for the text in Inkscape, but the package 'transparent.sty' is not loaded}%
    \renewcommand\transparent[1]{}%
  }%
  \providecommand\rotatebox[2]{#2}%
  \ifx\svgwidth\undefined%
    \setlength{\unitlength}{1200bp}%
    \ifx\svgscale\undefined%
      \relax%
    \else%
      \setlength{\unitlength}{\unitlength * \real{\svgscale}}%
    \fi%
  \else%
    \setlength{\unitlength}{\svgwidth}%
  \fi%
  \global\let\svgwidth\undefined%
  \global\let\svgscale\undefined%
  \makeatother%
  \begin{picture}(1,1)%
    \put(0,0){\includegraphics[width=\unitlength,page=1]{a05L28e-5i84e-7zoom.pdf}}%
    \put(0.272,0.47409524){\color[rgb]{0,0,0}\makebox(0,0)[lb]{\smash{$c_-$}}}%
  \end{picture}%
\endgroup%

%% file: a05L01-5zoom.pdf_tex
\begingroup%
  \makeatletter%
  \providecommand\color[2][]{%
    \errmessage{(Inkscape) Color is used for the text in Inkscape, but the package 'color.sty' is not loaded}%
    \renewcommand\color[2][]{}%
  }%
  \providecommand\transparent[1]{%
    \errmessage{(Inkscape) Transparency is used (non-zero) for the text in Inkscape, but the package 'transparent.sty' is not loaded}%
    \renewcommand\transparent[1]{}%
  }%
  \providecommand\rotatebox[2]{#2}%
  \ifx\svgwidth\undefined%
    \setlength{\unitlength}{1200bp}%
    \ifx\svgscale\undefined%
      \relax%
    \else%
      \setlength{\unitlength}{\unitlength * \real{\svgscale}}%
    \fi%
  \else%
    \setlength{\unitlength}{\svgwidth}%
  \fi%
  \global\let\svgwidth\undefined%
  \global\let\svgscale\undefined%
  \makeatother%
  \begin{picture}(1,1)%
    \put(0,0){\includegraphics[width=\unitlength,page=1]{a05L01-5zoom.pdf}}%
    \put(0.47714286,0.4638095){\color[rgb]{0,0,0}\makebox(0,0)[lb]{\smash{$c_-$}}}%
  \end{picture}%
\endgroup%

%% file: corbesbottcher.pdf_tex
\begingroup%
  \makeatletter%
  \providecommand\color[2][]{%
    \errmessage{(Inkscape) Color is used for the text in Inkscape, but the package 'color.sty' is not loaded}%
    \renewcommand\color[2][]{}%
  }%
  \providecommand\transparent[1]{%
    \errmessage{(Inkscape) Transparency is used (non-zero) for the text in Inkscape, but the package 'transparent.sty' is not loaded}%
    \renewcommand\transparent[1]{}%
  }%
  \providecommand\rotatebox[2]{#2}%
  \ifx\svgwidth\undefined%
    \setlength{\unitlength}{1325.78873892bp}%
    \ifx\svgscale\undefined%
      \relax%
    \else%
      \setlength{\unitlength}{\unitlength * \real{\svgscale}}%
    \fi%
  \else%
    \setlength{\unitlength}{\svgwidth}%
  \fi%
  \global\let\svgwidth\undefined%
  \global\let\svgscale\undefined%
  \makeatother%
  \begin{picture}(1,0.80978214)%
    \put(0,0){\includegraphics[width=\unitlength,page=1]{corbesbottcher.pdf}}%
    \put(0.42885517,0.16404258){\color[rgb]{0,0,0}\makebox(0,0)[lb]{\smash{$\gamma_0$}}}%
    \put(0.42471745,0.31472374){\color[rgb]{0,0,0}\makebox(0,0)[lb]{\smash{$0$}}}%
    \put(0.35834187,0.45557791){\color[rgb]{0,0,0}\makebox(0,0)[lb]{\smash{$B_a(U)$}}}%
    \put(0.53247001,0.55315865){\color[rgb]{0,0,0}\makebox(0,0)[lb]{\smash{$c_-$}}}%
    \put(0.37213422,0.57350232){\color[rgb]{0,0,0}\makebox(0,0)[lb]{\smash{$U$}}}%
    \put(0.58401889,0.67177268){\color[rgb]{0,0,0}\makebox(0,0)[lb]{\smash{$z_0$}}}%
    \put(0.63384366,0.59488046){\color[rgb]{0,0,0}\makebox(0,0)[lb]{\smash{$V$}}}%
    \put(0.76866365,0.47971449){\color[rgb]{0,0,0}\makebox(0,0)[lb]{\smash{$\gamma_0^{-1}$}}}%
    \put(0.79452429,0.32730923){\color[rgb]{0,0,0}\makebox(0,0)[lb]{\smash{$A^{-1}$}}}%
    \put(0.37041016,0.7609056){\color[rgb]{0,0,0}\makebox(0,0)[lb]{\smash{$\cercle$}}}%
  \end{picture}%
\endgroup%

%% file: esquema1.pdf_tex
\begingroup%
  \makeatletter%
  \providecommand\color[2][]{%
    \errmessage{(Inkscape) Color is used for the text in Inkscape, but the package 'color.sty' is not loaded}%
    \renewcommand\color[2][]{}%
  }%
  \providecommand\transparent[1]{%
    \errmessage{(Inkscape) Transparency is used (non-zero) for the text in Inkscape, but the package 'transparent.sty' is not loaded}%
    \renewcommand\transparent[1]{}%
  }%
  \providecommand\rotatebox[2]{#2}%
  \ifx\svgwidth\undefined%
    \setlength{\unitlength}{1433.14287109bp}%
    \ifx\svgscale\undefined%
      \relax%
    \else%
      \setlength{\unitlength}{\unitlength * \real{\svgscale}}%
    \fi%
  \else%
    \setlength{\unitlength}{\svgwidth}%
  \fi%
  \global\let\svgwidth\undefined%
  \global\let\svgscale\undefined%
  \makeatother%
  \begin{picture}(1,0.73365233)%
    \put(0,0){\includegraphics[width=\unitlength,page=1]{esquema1.pdf}}%
    \put(0.25007978,0.28229669){\color[rgb]{0,0,0}\makebox(0,0)[lb]{\smash{$A_0$}}}%
    \put(0.55773529,0.32519942){\color[rgb]{0,0,0}\makebox(0,0)[lb]{\smash{$T_0$}}}%
    \put(0.49696974,0.34066988){\color[rgb]{0,0,0}\makebox(0,0)[lb]{\smash{$0$}}}%
    \put(0,0){\includegraphics[width=\unitlength,page=2]{esquema1.pdf}}%
    \put(0.46092508,0.64912289){\color[rgb]{0,0,0}\makebox(0,0)[lb]{\smash{$A^*(\infty)$}}}%
    \put(0,0){\includegraphics[width=\unitlength,page=3]{esquema1.pdf}}%
    \put(0.86698569,0.53333341){\color[rgb]{0,0,0}\makebox(0,0)[lb]{\smash{$c_+$}}}%
    \put(0.69840515,0.43987247){\color[rgb]{0,0,0}\makebox(0,0)[lb]{\smash{$c_-$}}}%
    \put(0,0){\includegraphics[width=\unitlength,page=4]{esquema1.pdf}}%
    \put(0.74992029,0.46012768){\color[rgb]{0,0,0}\makebox(0,0)[lb]{\smash{$D_{0}$}}}%
    \put(0.47208935,0.25677838){\color[rgb]{0,0,0}\makebox(0,0)[lb]{\smash{$5-1$}}}%
    \put(0.5821372,0.54545459){\color[rgb]{0,0,0}\makebox(0,0)[lb]{\smash{$1-1$}}}%
    \put(0.39617229,0.51355668){\color[rgb]{0,0,0}\makebox(0,0)[lb]{\smash{$2-1$}}}%
  \end{picture}%
\endgroup%

%% file: esquemaanells.pdf_tex
\begingroup%
  \makeatletter%
  \providecommand\color[2][]{%
    \errmessage{(Inkscape) Color is used for the text in Inkscape, but the package 'color.sty' is not loaded}%
    \renewcommand\color[2][]{}%
  }%
  \providecommand\transparent[1]{%
    \errmessage{(Inkscape) Transparency is used (non-zero) for the text in Inkscape, but the package 'transparent.sty' is not loaded}%
    \renewcommand\transparent[1]{}%
  }%
  \providecommand\rotatebox[2]{#2}%
  \ifx\svgwidth\undefined%
    \setlength{\unitlength}{3214.85703125bp}%
    \ifx\svgscale\undefined%
      \relax%
    \else%
      \setlength{\unitlength}{\unitlength * \real{\svgscale}}%
    \fi%
  \else%
    \setlength{\unitlength}{\svgwidth}%
  \fi%
  \global\let\svgwidth\undefined%
  \global\let\svgscale\undefined%
  \makeatother%
  \begin{picture}(1,0.45005333)%
    \put(0,0){\includegraphics[width=\unitlength,page=1]{esquemaanells.pdf}}%
    \put(0.40277285,0.41215784){\color[rgb]{0,0,0}\makebox(0,0)[lb]{\smash{$A^*(\infty)$}}}%
    \put(0.23358628,0.21306878){\color[rgb]{0,0,0}\makebox(0,0)[lb]{\smash{$T_0$}}}%
    \put(0.2974598,0.27239235){\color[rgb]{0,0,0}\makebox(0,0)[lb]{\smash{$A_0$}}}%
    \put(0.38944645,0.23971411){\color[rgb]{0,0,0}\makebox(0,0)[lb]{\smash{$\mathcal{U}_c$}}}%
    \put(0,0){\includegraphics[width=\unitlength,page=2]{esquemaanells.pdf}}%
    \put(0.71722422,0.33577828){\color[rgb]{0,0,0}\makebox(0,0)[lb]{\smash{$B_{a,\lambda}(\mathcal{U}_c)$}}}%
    \put(0.74537775,0.22010717){\color[rgb]{0,0,0}\makebox(0,0)[lb]{\smash{$T_0$}}}%
    \put(0.91932655,0.41316015){\color[rgb]{0,0,0}\makebox(0,0)[lb]{\smash{$A^*(\infty)$}}}%
    \put(0,0){\includegraphics[width=\unitlength,page=3]{esquemaanells.pdf}}%
    \put(0.48596282,0.33573786){\color[rgb]{0,0,0}\makebox(0,0)[lb]{\smash{$4-1$}}}%
    \put(0,0){\includegraphics[width=\unitlength,page=4]{esquemaanells.pdf}}%
    \put(0.48093541,0.09140519){\color[rgb]{0,0,0}\makebox(0,0)[lb]{\smash{$3-1$}}}%
    \put(0.22855886,0.41114918){\color[rgb]{0,0,0}\makebox(0,0)[lb]{\smash{$\mathcal{V}_4$}}}%
    \put(0.22152047,0.33975979){\color[rgb]{0,0,0}\makebox(0,0)[lb]{\smash{$\mathcal{V}_3$}}}%
    \put(0.23056983,0.26133202){\color[rgb]{0,0,0}\makebox(0,0)[lb]{\smash{$\mathcal{V}_2$}}}%
    \put(0,0){\includegraphics[width=\unitlength,page=5]{esquemaanells.pdf}}%
    \put(0.29894277,0.36690787){\color[rgb]{0,0,0}\makebox(0,0)[lb]{\smash{$c_-$}}}%
    \put(0,0){\includegraphics[width=\unitlength,page=6]{esquemaanells.pdf}}%
    \put(0.36631021,0.29752946){\color[rgb]{0,0,0}\makebox(0,0)[lb]{\smash{$\mathcal{V}_1$}}}%
    \put(0.73733392,0.28445817){\color[rgb]{0,0,0}\makebox(0,0)[lb]{\smash{$\mathcal{W}_3$}}}%
    \put(0.7363284,0.3910395){\color[rgb]{0,0,0}\makebox(0,0)[lb]{\smash{$\mathcal{W}_4$}}}%
  \end{picture}%
\endgroup%